\documentclass{easychair}
\usepackage{prooftree,proof} 
\usepackage{amssymb,amsfonts} 

\def\O{\Box}
\def\P{\Diamond}
\def\L{\mathcal{L}}

\def\M{\mathcal{M}}
\def\to{\supset}

\def\To{\Longrightarrow}
\def\d{\Delta}
\def\g{\Gamma}

\def\de{\mathcal{D}}
\def\<{\langle}
\def\>{\rangle}
\def\gsd{\g\To\d}

\newcommand{\infrule}[1]{\scriptstyle\it{#1}}
\newcommand{\infruler}[1]{\scriptstyle{\textrm{#1}}}
\title{Sequent Calculi and Interpolation for Non-Normal Modal and Deontic Logics%
}

\author{
Eugenio Orlandelli
}

\institute{
  }
\authorrunning{Orlandelli}

\usepackage{tikz}
\usepgflibrary{arrows}
\usetikzlibrary{matrix,arrows,automata,decorations,fit,backgrounds}
\usetikzlibrary{decorations.pathreplacing,shapes.geometric,decorations.text}

\newtheorem{theorem}{Theorem}[section]  
\newtheorem{corollary}[theorem]{Corollary}
\newtheorem{lemma}[theorem]{Lemma}
\theoremstyle{definition}
\newtheorem{definition}[theorem]{Definition}
\newtheorem{proposition}[theorem]{Proposition}
\newtheorem{example}[theorem]{Example}
\titlerunning{Sequent calculi and interpolation for non-normal modal and deontic logics}

\usepackage{fancyhdr}
\begin{document}
\maketitle

\begin{abstract}   G3-style sequent calculi for the logics in the cube of non-normal modal logics and for their deontic extensions are studied. For each calculus we prove that weakening and contraction are height-preserving admissible, and we give a syntactic proof of the admissibility of cut. This implies that the subformula property holds  and that derivability can be decided by a terminating proof search whose complexity is in {\sc Pspace}. These calculi are shown to be equivalent to the axiomatic ones and, therefore,  they are sound and complete with respect to neighbourhood semantics.  Finally,  it is given a Maehara-style proof of Craig's interpolation theorem for most of the logics considered.
\end{abstract}
\section{Introduction}

For many interpretations  of the modal operators -- e.g., for deontic, epistemic,  game-theoretic, and high-probability interpretations -- it is necessary to adopt logics that are weaker than the normal ones; e.g., deontic paradoxes, see \cite{G13,M06},  are one of the main motivations for adopting a non-normal  deontic logic. Non-normal  logics, see \cite{C80} for naming conventions,  are quite well understood from a semantic point of view  by means of neighbourhood semantics  \cite{H09,P17}.
  Nevertheless, until recent years their proof theory has been rather limited  since it was mostly confined to Hilbert-style axiomatic systems. This situation seems to be rather unsatisfactory since it is difficult to find derivations in axiomatic systems. When the aim is to find derivations and to analyse their structural properties, sequent calculi are to be preferred to  axiomatic systems. Recently  different kinds of sequent calculi for non-normal logics have been proposed: Gentzen-style calculi  \cite{I05,I11,L00,O15}; labelled \cite{GM14,CO18} and display \cite{P19} calculi based on  translations into normal modal logics; labelled calculi based on the internalisation of neighbourhood  \cite{N17,NO19} and  bi-neighbourhood  \cite{DON} semantics;   and, finally, \mbox{linear nested sequents  \cite{L17}.}

 This paper, which extends the results  presented in \cite{O15},  concentrates on Gentzen-style calculi since they are better suited than labelled calculi, display calculi, and  nested sequents to give  decision procedures (computationally well-behaved) and constructive proofs of interpolation theorems. We consider cut- and contraction-free G3-style sequent calculi for all the logics in the cube of non-normal modalities and for their  extensions with the deontic axioms $D^\P:=\O A\to\Diamond A$ and $D^\bot:=\neg\O\bot$.  The calculi we present have the subformula property and  allow for a straightforward decision procedure by a terminating loop-free proof search. Moreover, with the  exception of the calculi for {\bf EC(N)} and its deontic extensions, they are \emph{standard} \cite{G16} -- i.e., each operator is handled by a finite number of rules with a finite number of premisses -- and they admit of a Maehara-style constructive proof of Craig's interpolation  theorem.
 
 This work improves on previous ones  on Gentzen-style calculi for non-normal logics  in that we  prove cut admissibility   for non-normal modal and deontic logics, and not only for the modal ones  \cite{L00,I05,I11}. Moreover, we prove height-preserving admissibility of weakening and contraction, whereas neither weakening nor contraction is admissible in \cite{L00,I05} and weakening but not contraction is admissible in \cite{I11}. The admissibility of contraction is a major improvement since, as it is well known, contraction can be as bad as cut for  proof search: we may continue to duplicate some formula forever  and, therefore,  we need a (computationally expensive) loop-checker to ensure termination.  Proof search  procedures based on contraction-free calculi terminate because the height of derivations is bounded by a number depending on the complexity of the end-sequent and, therefore, we avoid the need of  loop-checkers.  To illustrate, the introduction  of contraction-free calculi has allowed to give computationally optimal decision procedures for propositional intuitionistic logic  ({$\mathbf{IL_p}$) \cite{H93} and for the normal modal  logics {\bf K} and {\bf T} \cite{B97,H95}.   The existence of a loop-free terminating decision procedure has also allowed  to give a constructive proof of uniform interpolation for $\mathbf{IL_p}$ \cite{P92} as well as for  {\bf K} and {\bf T} \cite{B07}.  The cut- and contraction-free calculi for non-normal logics considered here are such that the height of each derivation is bounded by the weight of its end-sequent and, therefore, we easily obtain a polynomial space upper complexity bound for proof search. This upper bound  is  optimal for the  logics  having $C$  as theorem (the satisfiability problem for non-normal modal logics without $C$   is in  {\sc NP}, see \cite{V89}).
 
  Moreover,  the introduction of well-behaved calculi for non-normal deontic logics is interesting since  proof analysis can be applied to  the deontic paradoxes \cite{M06} that are one of the central topics of deontic reasoning. We illustrate this in Section  \ref{forrester} by considering Forrester's Paradox \cite{F84} and by showing that proof analysis cast  doubts on the widespread opinion \cite{M06,P17,T97} that Forrester's argument provides evidence against rule $RM$ (see Table \ref{rulesinf}).
If   Forrester's argument  is formalized as in \cite{M06} then  it does not compel us to adopt a deontic logic weaker than {\bf KD}. If, instead, it is formalised as in \cite{T97} then it forces the adoption of a logic where $RM$ fails, but the formal derivation  differs substantially from Forrester's informal argument. 
 
 It is also given a constructive proof of interpolation for all logics having a standard calculus. To our knowledge in the literature there is no other constructive  study of interpolation in non-normal  logics.  In \cite[Chap(s). 3.8 and 6.6]{F83} a constructive proof of Craig's (and Lyndon's)  interpolation theorem is given for the modal logics {\bf K} and {\bf R}, and for some of their extensions, including the deontic ones, but the proof makes use of model-theoretic notions. A proof of interpolation by the  Maehara-technique for {\bf KD} is given in \cite{V93}. For a thorough study of  interpolation in modal logics we refer the reader to \cite{G05}. A model-theoretic proof of interpolation for {\bf E} is given in \cite{H09}, and a coalgebraic proof  of (uniform) interpolation for all the logics considered here, as well as all other rank-1 modal logics (see below), is given in \cite{P13}. As it is explained in Example \ref{prob}, we have not been able to prove interpolation for calculi containing the non-standard rule $LR$-$C$ (see Table \ref{Modal rules}) and, as far as we know, it is still an open problem whether it is possible to give a constructive proof of interpolation for these logics.
 
 \paragraph{Related Work.}\label{related}
 The modal rules of inference  presented in Table \ref{Modal rules} are obtained from the rules presented in \cite{L00} by adding weakening contexts to the conclusion of the rules. This minor modification, used also in \cite{I11,P13,PS10} for several modal rules, allows us to shift from set-based sequents to multiset-based ones  and to prove not only that cut is admissible, as it is done in \cite{I05,I11,L00}, but also that weakening and contraction are height-preserving admissible. Given that implicit contraction is not eliminable from set-based sequents, the decision procedure for non-normal logics given in \cite{L00} is based on  a model-theoretic inversion technique so that it is possible to define a  procedure that  outputs a derivation for all valid sequents and a finite   countermodel  for all invalid ones. One weakness of this decision  procedure is that  it does not respect the subformula property for logics without rule $RM$ (the procedure  adds instances of the excluded middle).

The paper \cite{I05} considers multiset-based calculi for the non-normal logic  {\bf M(N)} and for its extensions with axioms $D^\P,T, 4, 5$, and $B$. Nevertheless,  neither weakening nor contraction is eliminable  because there are no weakening contexts in the conclusion of the modal rules. In \cite{I11} multiset-based sequent calculi for the non-normal  logic {\bf E(N)} and  for its extensions with  axioms $D^\P,T$, 4, 5, and $B$ are given. The rules $LR$-$E$ and $R$-$N$ are  as in Table \ref{Modal rules}, but the deontic axiom  $D^\P$ is expressed  by the following  rule:

$$
\infer[\infrule D\text{-}2]{\O A,\O B,\g\To\d}{A,B\To&(\To A,B)}
$$
where the right premiss  is present when we are working over $LR$-$E$ and  it has to be omitted when we work over $LR$-$M$. In the calculi in \cite{I05,I11} weakening and contraction are taken as primitive rules and not as admissible ones as in the present approach. Even if it is easy to show that weakening is eliminable from the calculi in \cite{I11}, contraction cannot be eliminated because rule \emph{D-2} has exactly two principal formulas and, therefore, it is not possible to permute contraction up with respect to instances of  rule \emph{D-2} (see Theorem \ref{contr}). 
The presence of a non-eliminable rule of contraction makes the elimination of cut more problematic: in most cases we cannot  eliminate the cut directly, but we have to consider the rule known as multicut \cite[p. 88]{NP01}.  Moreover, cut is not eliminable from the calculus given in \cite{I11} for the deontic  logic {\bf END}. The formula  $D^\bot:= \neg\O\bot$ is a theorem of this logic, but it can be derived only with a non-eliminable instance of  cut  as in:

$$
\infer[\infrule R\neg]{\To \neg\O\bot}{
\infer[\infrule Cut]{\O\bot\To}{
\infer[\infrule R\mbox{-}N]{\To\O\top}{\To\top}&
\infer[\infrule D\mbox{-}2]{\O\top,\O\bot\To}{\bot,\top\To&\To\bot,\top}}}
$$ 

Finally, it is worth noticing that all the non-normal logics we consider here are  \emph{rank-1} logics in the sense of \cite{P13,PS10,PS09} -- i.e., logics whose modal axioms are propositional combinations of formulas of shape $\Box\phi$, where  $\phi$ is purely propositional -- and the calculi we give for the  modal logics {\bf E}, {\bf M}, {\bf K} and {\bf KD} are explicitly considered in \cite{P13,PS09}.  Thus, they are part of the family of modal coalgebraic logics  \cite{P13,PS10,PS09} and most of the results in this paper can be seen as  instances of general results that hold for rank-1 (coalgebraic) logics. If, in particular, we consider cut-elimination for coalgebraic logics \cite{PS10} then all our calculi  absorb congruence  and Theorem \ref{contr} and case 3 of Theorem \ref{cut} show that they absorb  contraction and cut. Hence, \cite[Thm. 5.7]{PS10} entails  that cut and contraction are   admissible in these calculi; moreover, \cite[Props. 5.8 and 5.11]{PS10} entail that they are one-step cut free complete w.r.t. coalgebraic semantics. This latter result gives a semantic proof of cut admissibility  in the calculi considered here. Analogously, if we consider decidability, the polynomial space upper bound we find in Section \ref{decision} coincides with that found in \cite{PS09} for rank-1 modal logics.

\paragraph{Synopsis. }
Section \ref{secaxiom} summarizes the basic notions of axiomatic systems and of neighbourhood semantics for  non-normal  logics.  Section \ref{seccalculi} presents G3-style  sequent calculi for these logics and then shows that weakening and contraction are height-preserving admissible and that cut is (syntactically) admissible.  Section \ref{secdecax} describes a  terminating proof-search decision procedure for all calculi, it shows that each calculus is equivalent to the  corresponding axiomatic system, and  it applies proof search to Forrester's paradox.  Finally, Section \ref{secinterpol} gives a Maehara-style    constructive proof of Craig's interpolation theorem for the logics having a standard calculus.
\section{Non-normal Logics}\label{secaxiom}
\subsection{Axiomatic Systems}
We introduce, following \cite{C80},  the basic notions  of non-normal  logics. Given a countable set of propositional variables $\{p_n\,|\,n\in \mathbb{N}\}$, the formulas of the modal language $\L$  are generated by:

$$
A::= \;p_n\;|\;\bot\;|\;A\wedge A\;|\;A\lor A\;|\;A\to A\;|\;\O A
$$
We remark that $\bot$ is a 0-ary logical symbol. This will be extremely important in the proof of Craig's interpolation theorem. As usual $\neg A$ is a shorthand for $A\to\bot$, $\top$ for $\bot\to\bot$, $A \leftrightarrow B$  for $(A\to B)\wedge(B\to A)$, and $\P A$ for $\neg\O\neg A$. We follow the usual conventions for parentheses.

\begin{table}
\caption{ Rules of inference}\label{rulesinf}
\begin{center}
\begin{tabular}{ccc}
\hline\hline\noalign{\smallskip}
\infer[\infrule RE]{\O A\leftrightarrow\O B}{A\leftrightarrow B}
&$\quad$&
\infer[\infrule RM]{\O A\to\O B}{A\to B}
\\\noalign{\smallskip\smallskip\smallskip}
\infer[\infrule RR,\; n\geq 1]{(\O A_1\wedge\dots\wedge\O A_n)\to\O B}{(A_1\wedge\dots\wedge A_n)\to B}
&$\quad$&
\infer[\infrule RK,\; n\geq 0]{(\O A_1\wedge\dots\wedge\O A_n)\to\O B}{(A_1\wedge\dots\wedge A_n)\to B}
\\
\noalign{\smallskip}\hline\hline
\end{tabular}
\end{center}
\end{table}

\begin{table}
\caption{ Axioms}\label{axioms}
\begin{center}
\begin{tabular}{cclccclcccl}
\hline\hline\noalign{\smallskip}
$M$)&$\quad$&$\O (A\wedge B)\to(\O A\wedge\O B)$&${}$\qquad{}\qquad{}\qquad&$C$)&$\quad$&$(\O A\wedge\O B)\to\O(A\wedge B)$ \\\noalign{\smallskip\smallskip\smallskip}$N$)&$\quad$&$\O \top$&${}$\qquad{}\qquad{}\qquad&
$D^\bot$)&$\quad$&$\neg\O\bot$ \\\noalign{\smallskip\smallskip\smallskip}$D^\P$)&$\quad$&$ \O A\to\P A$
\\
\noalign{\smallskip}\hline\hline
\end{tabular}
\end{center}
\end{table}

Let {\bf L}  be the logic containing  all $\L$-instances of propositional tautologies as axioms, and   modus ponens ($MP$) as inference rule.  The minimal non-normal modal logic {\bf E} is the logic {\bf L} plus the rule $RE$ of Table \ref{rulesinf}. We will consider all the logics that are obtained by extending {\bf E} with some set of axioms from Table \ref{axioms}. We will denote the logics according to the axioms that define them, e.g. {\bf EC} is the logic  {\bf E}$\,\oplus \,C$, and {\bf EMD$^\bot$} is {\bf E}$\,\oplus\, M\oplus D^\bot$.  By {\bf X} we denote any of these logics and we  write \mbox{{\bf X} $\vdash A$} whenever $A$ is a theorem of {\bf X}. We will call \emph{modal} the logics containing neither $D^\bot$ nor $D^\P$, and \emph{deontic} those containing at least one of them. We have followed the usual naming conventions for the modal axioms, but we have introduced new  conventions for the deontic ones: $D^\bot$ is usually called either  $CON$ or $P$ and $D^\P$ is usually called $D$, cf. \cite{G00,G13,M06}.
 
 It is also possible to give an equivalent  rule-based axiomatization of some of these logics. In particular, the logic {\bf EM}, also called {\bf M}, can be axiomatixed as {\bf L} plus the rule $RM$ of Table \ref{rulesinf}.  The logic  {\bf EMC}, also called {\bf R}, can be axiomatized as {\bf L} plus the rule  $RR$ of Table \ref{rulesinf}. Finally,  the logic {\bf EMCN}, i.e. the smallest normal modal logic {\bf K},  can be axiomatized as {\bf L} plus the rule $RK$ of Table \ref{rulesinf}. These rule-based axiomatizations will be useful later on since they simplify the proof of the equivalence between  axiomatic systems and  sequent calculi (Theorem \ref{comp}).

The following proposition states the well-known relations between the theorems of non-normal modal  logics. For a proof  the reader is referred to \cite{C80}.

\begin{proposition} For any formula $A\in \L$ we have that {\bf E} $\vdash A$ implies {\bf M}  $\vdash A$; {\bf M} $\vdash A$   implies {\bf R} $\vdash A$; {\bf R} $\vdash A$  implies {\bf K} $\vdash A$. Analogously for the logics containing axiom $N$ and/or axiom $C$.\end{proposition}

Axiom $D^\bot$ is {\bf K}-equivalent to $D^\P$, but the correctness of $D^\P$ has been a big issue in the literature on deontic logic. This fact urges the study of  logics weaker than {\bf KD}, where $D^\bot$ and $D^\P$ are no more equivalent \cite{C80}. 
The deontic  formulas $D^\bot$ and $D^\P$ have the following relations in the logics we are considering.
\begin{proposition} $D^\bot$ and $D^\P$ are  independent  in {\bf E}; $D^\bot$ is derivable from $D^\P$ in  non-normal logics containing at least one of the axioms  $M$  and $N$; $D^\P$ is derivable from $D^\bot$ in non-normal logics containing  axiom $C$.
\end{proposition}
In Figure  \ref{cube} the reader finds the lattice of non-normal modal logics, see \cite[p. 237]{C80}, and in Figure \ref{cubed}  the lattice of non-normal deontic logics. 
 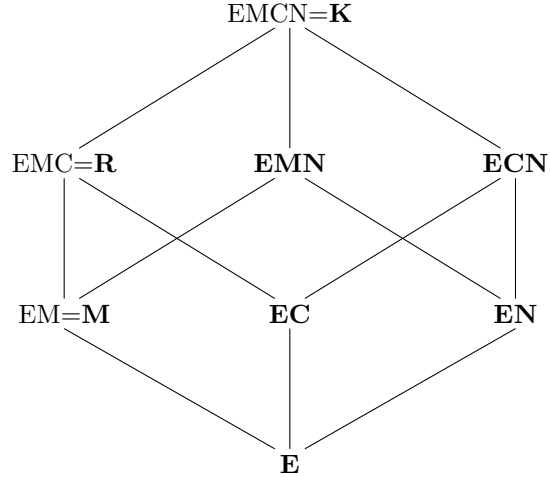
\begin{figure}[t]
\begin{center}
\begin{tikzpicture}
\node at (1,-2)  {EM={\bf M}};
\node at (4,-4)  {{\bf E}};
\node at (4,-2)  {{\bf EC}};
\node at (7,-2)  {{\bf EN}};
\node at (1,0)  {EMC={\bf R}};
\node at (4,2)  {EMCN={\bf K}};
\node at (4,0)  {{\bf EMN}};
\node at (7,0)  {{\bf ECN}};

\draw (3.8,-3.8) -- (1,-2.2);
\draw (4,-3.8) -- (4,-2.2); 
\draw (4.2,-3.8) -- (7,-2.2);

\draw (1,-1.8) -- (1,-0.2);
\draw (4,0.2) -- (4,1.8); 
\draw (7,-1.8) -- (7,-0.2);

\draw (1.2,0.2) -- (3.8,1.8);
\draw (6.8,0.2) -- (4.2,1.8);
\draw (1.2,-1.8) -- (3.8,-0.2);
\draw (6.8,-1.8) -- (4.2,-0.2);

\draw (3.8,-1.8) -- (1.2,-0.2);
\draw (4.2,-1.8) -- (6.8,-0.2);
\end{tikzpicture}
\caption{Lattice of non-normal modal logics}\label{cube}
\end{center}
\end{figure}

 \begin{figure}[t]
\begin{center}
\scalebox{0.80000}{\begin{tikzpicture}
\filldraw  [black] (5,-4.2) circle (4pt)  
		     (6.7,-4.2) circle (4pt) 
		     (5,-3) circle (4pt); 
\node[below] at (4.8,-4.3) {\small{\bf ED$^\bot$}};
\node[below] at (6.7,-4.3)  {\small{\bf ED$^\P$}};
\node[left] at (4.9,-3) {\small{ ED$^\bot$D$^\P$}=};
\node[left] at (4.9,-3.3) {\small{\bf ED}};
\draw (5,-4.2) -- (1,-1.7);
\draw (5,-3) -- (1,-0.5);
\draw (5,-4.2) -- (5,-3);
\draw (6.7,-4.2) -- (5,-3);
\draw  (5,-3) -- (5,-0.5);
\draw (6.7,-4.2) -- (6.7,-1.6);
\draw (5,-4.2) -- (9,-1.7);
\draw (5,-3) -- (9,-0.5);


\filldraw  [black] (6.7,-1.5) circle (4pt)  
		     (5,-0.5) circle (4pt); 
\node[left] at (6.5,-1.5)  {\small{\bf ECD$^\P$}};
\node[left] at (4.9,-0.5)  {\small{ ECD$^\bot$=}};
\node[left] at (4.9,-0.8)  {\small{ \bf ECD}};
\draw (5,-0.5) -- (9,2);
\draw (6.7,-1.5)-- (5,-0.5);
\draw (5,-0.5) -- (1,2);

\filldraw  [black] (9,-1.7) circle (4pt)  
		     (9,-0.5) circle (4pt); 
\node[right] at (9.2,-1.7)  {\small{\bf END$^\bot$}};
\node[right] at (9.1,-0.5)  {\small{ END$^\P$}= {\bf END}};
\draw (9,-0.5) -- (9,2);
\draw (9,-1.7) -- (9,-0.5);
\draw (9,-1.7) -- (5,0.8);
\draw (9,-0.5) -- (5,2);

\filldraw  [black] (1,2) circle (4pt)  ;
\node[left] at (0.9,2)  {\small{{ RD$^\bot$}= {RD$^\P$}= {\bf RD}}};
\draw (1,2) -- (5,4.5);

\filldraw  [black] (5,4.5) circle (4pt)  ;
\node[above] at (5,4.6)  {\small{{KD$^\bot$}= {KD$^\P$}= {\bf KD}}};

\filldraw  [black] (1,-1.7) circle (4pt)  
		     (1,-0.5) circle (4pt); 
\node[left] at (0.9,-1.7)  {\small{\bf MD$^\bot$}};
\node[left] at (0.9,-0.5) {\small{ MD$^\P$}= {\bf MD}};
\draw (1,-1.7) -- (5,0.8);
\draw (1,-0.5) -- (5,2);
\draw (1,-1.7) -- (1,-0.5);
		     (5,2) circle (4pt); 
\draw (1,-0.5) -- (1,2);

\filldraw  [black] (5,0.8) circle (4pt)  
		    (5,2) circle (4pt);  
\node[right] at (5.2,0.8)  {\small{\bf MND$^\bot$}};
\node[right] at (5.1,2)  {\small{ MND$^\P$}= {\bf MND}};
\draw (5,0.8) -- (5,2); 
\draw (5,2) -- (5,4.5); 

\filldraw  [black] 
		     (9,2) circle (4pt); 
\node[right] at (9.1,2)  {\small{ECND$^\bot$} = {ECND$^\P$}= {\bf ECND}};
\draw (9,2) -- (5,4.5);
\end{tikzpicture}}
\caption{Lattice of non-normal deontic logics}\label{cubed}
\end{center}
\end{figure}
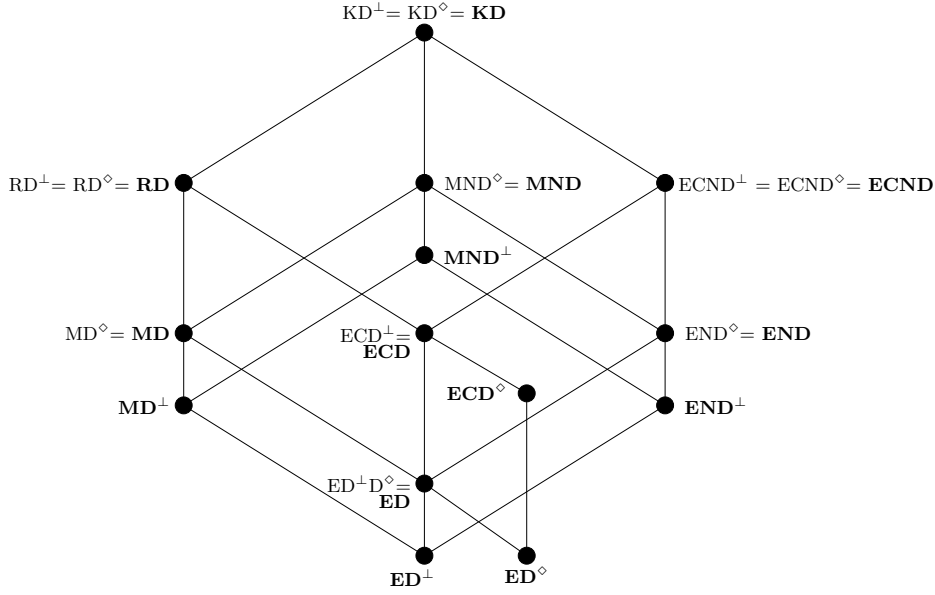

\subsection{Semantics}\label{semantics}
The most widely known semantics for non-normal logics is neighbourhood semantics. We sketch its main tenets following \cite{C80}, where neighbourhood models are called \emph{minimal models}. 

\begin{definition}A \emph{neighbourhood model} is a triple $\M:=\langle W,\, N,\,P\rangle$, where $W$ is a non-empty set of possible worlds; $N:W\longrightarrow 2^{2^W}$ is a neighbourhood function that associates to each possible world $w$ a set $N(w)$ of subsets   of $W$; and $P$ gives a truth value to each propositional variable at each world. \end{definition}

\noindent The definition of truth of a formula $A$ at a world $w$ of a neighbourhood model $\mathcal{M}$ -- $\models_w^\mathcal{M}A$ -- is the standard one for the classical connectives with the addition of

$$
\models_w^\mathcal{M}\O A\qquad \textnormal{iff}\qquad || A||^\mathcal{M}\in N(w)
$$
where $||A||^\mathcal{M}$ is the truth set of $A$ -- i.e., $||A||^\mathcal{M}=\{w\,|\,\models_w^\M A\}$. We say that a formula $A$  is \emph{valid} in a class $\mathcal{C}$ of neighbourhood models iff it is true in every world of every $\mathcal{M}\in\mathcal{C}$.

In order to give  soundness and completeness results for non-normal modal and deontic  logics with respect to (classes of) neighbourhood models, we introduce the following definition.
\begin{definition} Let $\mathcal{M}=\langle W,\, N,\,P\rangle$ be a neighbourhood model, $X,Y \in 2^W$, and $w\in W$, we say that: 
\begin{itemize}
\item $\M$ is \emph{supplemented}    if $X\cap Y\in N(w)$ imples $X\in N(w)$ and $Y\in N(w)$;
\item$\M$ is \emph{closed under finite intersection}  if $X\in N(w)$ and $Y\in N(w)$ imply $X\cap Y\in N(w)$;
\item $\M$ \emph{contains the unit} if   $W\in N(w)$;
\item $\M$ is \emph{non-blind}  if $X\in N(w)$ implies $X\neq \emptyset$; 
\item $\M$ is \emph{complement-free} if $X\in N(w)$ implies $W-X\not\in N(w)$.            
\end{itemize}\end{definition}


\begin{proposition}\label{corrax} We have the following correspondence results between $\L$-formulas and the properties of the neighbourhood function defined above:
\begin{itemize}
\item Axiom $M$ corresponds to supplementation;
\item Axiom $C$ corresponds to closure under finite intersection;
\item Axiom $N$ corresponds to containment of the unit;
\item Axiom $D^\bot$ corresponds to non-blindness;
\item Axiom $D^\P$ corresponds to complement-freeness.
\end{itemize}
\end{proposition}

\begin{theorem}\label{compax} {\bf E} is sound and complete with respect to the class of all neighbourhood models. Any logic {\bf X} which is obtained by extending {\bf E} with some axioms from Table \ref{axioms} is sound and complete with respect to the class of all neighbourhood models which satisfies all the properties corresponding to the axioms of {\bf X}.
\end{theorem}

See \cite{C80} for the proof of Proposition \ref{corrax} and of Theorem \ref{compax}.

\section{Sequent Calculi}\label{seccalculi}

 We introduce  sequent calculi for non-normal logics that extend the multiset-based  sequent calculus  {\bf G3cp} \cite{NP01,NP11,T00} for classical propositional logic  -- see Table \ref{G3cp} -- by adding some modal and deontic rules from Table \ref{Modal rules}.   In particular, we  consider the modal sequent calculi given in Table \ref{modcalculi}, which will be shown to capture the modal logics of Figure \ref{cube}, and their deontic extensions given in Table \ref{deoncalculi}, which will be shown to capture all deontic logics of Figure \ref{cubed}.
We adopt the following notational conventions: we use {\bf G3X} to denote a generic calculus from either Table \ref{modcalculi} or Table \ref{deoncalculi}, and we use {\bf G3Y(Z)}  to denote both {\bf G3Y} and {\bf GRYZ}.  All the rules in Tables \ref{G3cp} and  \ref{Modal rules} but  $LR$-$C$ and $L$-$D^{\P_C}$ are standard rules in the sense of \cite{G16}: each of them is a  single rule with a fixed number of premisses; $LR$-$C$ and $L$-$D^{\P_C}$, instead, stand for a recursively enumerable set of rules with a variable number of premisses.   

 For an introduction to {\bf G3cp} and  the relevant notions, the reader is referred to \cite[Chapter 3]{NP01}.
   We sketch here the main notions that will be used in this paper.   A \emph{sequent} is an expression  $\g\To \d$, where $\g$ and $\d$ are finite, possibly empty, multisets of formulas. If $\Pi$ is the (possibly empty) multiset $A_1,\dots,A_m$ then $\O\Pi$ is  the (possibly empty) multiset $\O A_1,\dots,\O A_m$. A \emph{derivation} of a sequent $\g\To\d$ in {\bf G3X} is an upward growing  tree of sequents having $\g\To\d$ as root, initial sequents or instances of rule $L\bot$ as leaves, and such that each non-initial node is the conclusion of an instance of one  rule of {\bf G3X} whose premisses are its children. In the rules in Tables \ref{G3cp} and \ref{Modal rules}, the multisets $\g$ and $\d$ are called \emph{contexts}, the other formulas occurring in the conclusion (premiss(es), resp.) are called \emph{principal} (\emph{active}). In a sequent the \emph{antecedent} (\emph{succedent}) is the multiset occurring to the left (right) of the sequent arrow $\To$.  As for {\bf G3cp}, a sequent $\g\To\d$ has the following  \emph{denotational interpretation}: the conjunction of the formulas in $\g$ implies the disjunction of the formulas in $\d$.
 
As measures for inductive proofs we  use the weight of a formula and the height of a derivation. The \emph{weight} of a formula $A$, $w(A)$, is defined inductively as follows: $w(\bot)=w(p_i)=0$; $w(\O A)=w(A)+1$; $w(A\circ B)=w(A)+w(B)+1$ (where    $\circ$ is one of the binary connectives $\wedge,\,\lor,\,\to$). The \emph{weight} of a sequent is the sum of the weight of the formulas occurring in that sequent.  The \emph{height} of a derivation is the length of its longest branch minus one. A rule of inference is said to be (\emph{height-preserving}) \emph{admissible} in {\bf G3X} if, whenever its premisses are derivable in {\bf G3X}, then also its conclusion is derivable (with at most the same derivation height) in {\bf G3X}. The \emph{modal depth} of a formula (sequent) is the maximal number of nested modal operators occurring in it(s members).
  

   \begin{table}
\caption{ The sequent calculus {\bf G3cp}}\label{G3cp}
\begin{center}
\scalebox{0.85000}{\begin{tabular}{cccc}
\hline\hline\noalign{\smallskip}
 Initial sequents: &$\qquad p_n,\g\To\d,p_n$&& $p_n$ propositional variable \\
 \noalign{\smallskip}\hline\noalign{\smallskip\smallskip}
  Propositional rules:&
  \infer[\infrule L\wedge]{A\wedge B,\gsd}{A,B,\gsd}

&$\qquad$&
\infer[\infrule R\wedge]{\gsd,A\wedge B}{\gsd,A\quad&\gsd,B}

\\\noalign{\smallskip\smallskip}
\infer[\infrule L\bot]{\bot,\gsd}{}&
\infer[\infrule L\lor]{A\lor B,\gsd}{A,\gsd\quad&B,\gsd}

&&
\infer[\infrule R\lor]{\gsd,A\lor B}{\gsd,A,B}

\\\noalign{\smallskip\smallskip}&
\infer[\infrule L\to]{A\to B,\gsd}{\gsd,A\quad&B,\gsd}

&&
\infer[\infrule R\to]{\gsd,A\to B}{A,\gsd,B}

\\
\noalign{\smallskip}\hline\hline
\end{tabular}}
\end{center}
%
\caption{Modal and deontic rules}\label{Modal rules}
\begin{center}
\scalebox{0.87000}{\begin{tabular}{llll}
\hline\hline\noalign{\smallskip}

\infer[\infrule LR\mbox{-}E]{\O A,\gsd,\O B}{A\To B\quad&B\To A}
&
\infer[\infrule LR\mbox{-}M ]{\O A,\gsd,\O B}{A\To B}
&
\multicolumn{2}{c}{\infer[\infrule LR\mbox{-}R]{\O A,\O \Pi,\gsd,\O B}{A,\Pi\To B}}

\\\noalign{\smallskip\smallskip}
\multicolumn{2}{c}{\infer[\infrule LR\mbox{-}C ]{\O A_1,\dots,\O A_n,\gsd,\O B}{A_1,\dots,A_n\To B&B\To A_1&{}^{\dots}&B\To A_n}
}
&
\infer[\infrule LR\mbox{-}K]{\O\Pi,\gsd,\O B}{\Pi\To B}&
\infer[\infrule R\mbox{-}N]{\gsd,\O B}{\To B}
\\\noalign{\smallskip}\hline\noalign{\smallskip}
\end{tabular}}

\scalebox{0.8700}{\begin{tabular}{llllll}

&
\infer[\infrule L\mbox{-}D^\bot]{\O A,\gsd}{A\To}
&
\infer[\infrule L\mbox{-}D^{\P_E}, \,|\Pi|\leq 2]{\O\Pi,\gsd}{\Pi\To&\To\Pi}
&
\infer[\infrule L\mbox{-}D^{\P_M}, \,|\Pi|\leq 2]{\O\Pi,\gsd}{\Pi\To}
\\\noalign{\smallskip\smallskip}
\phantom{aaaaaaaa}&
\multicolumn{2}{c}{\infer[\infrule L\mbox{-}D^{\P_C}]{\O \Pi,\O\Sigma,\gsd}{\Pi,\Sigma\To&\{ \To A, B|\; A\in\Pi, B\in \Sigma \}}
}
&
\infer[\infrule L\mbox{-}D^*]{\O\Pi,\gsd}{\Pi\To}&\phantom{aaaaaaa}
\\
\noalign{\smallskip}\hline\hline
\end{tabular}}\end{center}
%
 \caption{Modal  sequent calculi (\checkmark= rule of the calculus, $\star$ = admissible rule, $-$ = neither)}\label{modcalculi}
 
 \begin{center}
\scalebox{0.66000}{ \begin{tabular}{r|c|c|c|c|c|c|c|c|c|c|c|c|c|c|}  
\noalign{\smallskip} \hline\hline\noalign{\smallskip\smallskip}
&{\bf G3E}&{\bf G3EN}&{\bf G3M}& {\bf G3MN}&{\bf G3C}&{\bf G3CN}& {\bf G3R}&{\bf G3K}\\

 \noalign{\smallskip}\hline\noalign{\smallskip\smallskip}
$LR$-$E$&\checkmark&\checkmark&$\star$&$\star$&$\star$&$\star$&$\star$&$\star$\\

 \noalign{\smallskip}\hline\noalign{\smallskip\smallskip}
$LR$-$M$&$-$&$-$&\checkmark&\checkmark&$-$&$-$&$\star$&$\star$\\

 \noalign{\smallskip}\hline\noalign{\smallskip\smallskip}
$LR$-$C$&$-$&$-$&$-$&$-$&$\checkmark$&$\checkmark$&$\star$&$\star$\\

 \noalign{\smallskip}\hline\noalign{\smallskip\smallskip}
$LR$-$R$&$-$&$-$&$-$&$-$&$-$&$-$&\checkmark&$\star$\\

 \noalign{\smallskip}\hline\noalign{\smallskip\smallskip}
$LR$-$K$&$-$&$-$&$-$&$-$&$-$&$-$&$-$&\checkmark\\

 \noalign{\smallskip}\hline\noalign{\smallskip\smallskip}
$R$-$N$&$-$&$\checkmark$&$-$&$\checkmark$&$-$&$\checkmark$&$-$&$\star$\\\noalign{\smallskip}\hline\hline
  \end{tabular}}\end{center}
\caption{Deontic sequent calculi (\checkmark= rule of the calculus, $\star$ = admissible rule, $-$ = neither)}\label{deoncalculi}
\begin{center}
\scalebox{0.66000}{\begin{tabular}{l|c|c|c|c|c|c|c|c|c|c|}

\hline\hline\noalign{\smallskip\smallskip\smallskip}
  &{\bf G3E(N)D$^\bot$}&{\bf G3ED$^\P$}&{\bf G3E(N)D}&{\bf G3M(N)D$^\bot$}&{\bf G3M(N)D}&{\bf G3CD$^\P$}&{\bf G3C(N)D}&{\bf G3RD}&{\bf G3KD}\\
  
     \noalign{\smallskip}\hline\noalign{\smallskip\smallskip}
  $L$-$D^\bot$&\checkmark&$-$&\checkmark&\checkmark&$\star$&$-$&$\star$&$\star$&$\star$\\
  
   \noalign{\smallskip}\hline\noalign{\smallskip\smallskip}
  $L$-$D^{\P_E}$&$-$&\checkmark&\checkmark&$-$&$\star$&$\star$&$\star$&$\star$&$\star$\\
  
    \noalign{\smallskip}\hline\noalign{\smallskip\smallskip}
  $L$-$D^{\P_M}$&$-$&$-$&$-$&$-$&\checkmark&$\star$&$\star$&$\star$&$\star$\\
  
   \noalign{\smallskip}\hline\noalign{\smallskip\smallskip}
  $L$-$D^{\P_C}$&$-$&$-$&$-$&$-$&$-$&\checkmark&$\star$&$\star$&$\star$\\
  
   \noalign{\smallskip}\hline\noalign{\smallskip\smallskip}
  $L$-$D^*$&$-$&$-$&$-$&$-$&$-$&$-$&\checkmark&\checkmark&\checkmark
   
\\
\noalign{\smallskip}\hline\hline
\end{tabular}}
\end{center}
\end{table}

\subsection{Structural rules of inference}
We are now going to prove that the calculi {\bf G3X} have the same good structural properties of {\bf G3cp}: weakening and contraction are height-preserving admissible  and  cut is admissible. All proofs are extension of those for {\bf G3cp}, see  \cite[Chapter 3]{NP01}; in most cases, the modal rules have to be treated differently from the propositional ones  because of the presence of empty contexts in the premiss(es) of the modal ones. We adopt the following notational convention: given a derivation tree $\de_k$, the derivation tree of the $n$-th leftmost premiss of its last step is  denoted by $\de_{kn}$. We begin by showing that the restriction to atomic initial sequents, which is needed to have the propositional rules invertible, is not limitative in that initial  sequents with arbitrary principal  formula are derivable in {\bf G3X}.

\begin{proposition}\label{genax} Every instance of $A,\g\To\d, A$ is derivable in {\bf G3X}.\end{proposition}

\begin{proof} By induction on the weight of $A$. If $w(A)=0$ -- i.e., $A$ is atomic or $\bot$ -- then we have an instance of an initial sequent or of a conclusion of $L\bot$ and there is nothing to prove. If $w(A)\geq1$,  we  argue by cases according to the construction of $A$. In each case we apply, root-first, the appropriate rule(s) in order to obtain sequents where some proper subformula of $A$ occurs both in the antecedent and in the succedent. The claim then holds by the inductive hypothesis (IH). To illustrate, if $A\equiv\O B$ and we are in {\bf G3M(ND)}, we have:

$$
\infer[\infrule LR\mbox{-}M]{\O B,\gsd,\O B}{\infer[\infrule IH]{B\To B}{}}
$$
  \end{proof}

\begin{theorem}\label{weak} The left and right rules of weakening are height-preserving admissible in {\bf G3X}\

$$
\infer[\infrule LW]{A,\gsd}{\gsd}
\qquad
\infer[\infrule RW]{\gsd,A}{\gsd}
$$
  \end{theorem}

\begin{proof}The proof is a straightforward induction on the height of the derivation $\de$ of $\g\To \d$. 
 If the last step of $\de$ is by a propositional rule, we have to apply the same rule to the weakened premiss(es), which are derivable by IH, see \cite[Thm. 2.3.4]{NP01}. If it is by a modal or deontic rule, we proceed by adding $A$ to the appropriate weakening context of the conclusion of that rule  instance. To illustrate, if the last rule is $LR$-$E$, we transform 
 
$$
\infer[\infrule LR\mbox{-}E]{\O B,\gsd,\O C}{\deduce{B\To C}{\vdots\;\de_1}&\deduce{C\To B}{\vdots\;\de_2}  }
\qquad\mbox{into}\qquad
\infer[\infrule LR\mbox{-}E]{\O B,A,\gsd,\O C}{\deduce{B\To C}{\vdots\;\de_1}&\deduce{C\To B}{\vdots\;\de_2}  }
$$

{}\end{proof}

\noindent Before considering contraction, we recall some facts that will be useful later on. 
\begin{lemma} In {\bf G3X} the rules$\quad$
\infer[\infrule L\neg]{\neg A,\gsd}{\gsd,A}
and$\quad$
\infer[\infrule R\neg]{\gsd,\neg A}{A,\gsd}
are admissible.
\end{lemma}
\begin{proof}
We have the following derivations (the step by $RW$ is admissible thanks to Theorem \ref{weak}):

$$\infer[\infrule L\to]{A\to\bot,\gsd}{\gsd,A&\infer[\infrule L\bot]{\bot,\gsd}{}}
\qquad\qquad
\infer[\infrule R\to]{\gsd,A\to\bot}{\infer[\infrule RW]{A,\gsd,\bot}{A,\gsd}}
$$
\end{proof}

\begin{lemma}\label{inv}All propositional rules are height-preserving invertible in {\bf G3X}, that is the derivability of (a possible instance of) a conclusion of a propositional rule entails the derivability, with at most the same derivation height, of its premiss(es).\end{lemma}

\begin{proof} We have only to extend the proof  for {\bf G3cp}, see  \cite[Thm. 3.1.1]{NP01}, with new cases for the modal and deontic rules. If $A\circ B$ occurs in the antecedent (succedent) of the conclusion of an instance of a  modal or deontic rule then it must be a member of the weakening context $\g$ ($\d$)  of this rule instance, and we have only to change the weakening context according to the rule we are inverting.\end{proof}

\begin{theorem}\label{contr}The left and right rules of contraction are height-preserving admissible in {\bf G3X}\vspace{0.3cm}

$$
\infer[\infrule LC]{A,\gsd}{A,A,\gsd}
\qquad
\infer[\infrule RC]{\gsd,A}{\gsd,A,A}
$$
 \end{theorem}
 
 \begin{proof} The proof is by simultaneous  induction on the height  of the derivation $\de$  of the  premiss for left and right  contraction. The base case is straightforward. For the inductive steps, we have different strategies according to whether the last step in $\de$ is by a propositional  rule or not. If the last step in $\de$ is  by a propositional rule, we have two subcases: if the contraction formula is not principal in that step, we apply the inductive hypothesis and then the rule. Else  we start by using the height-preserving invertibility --  Lemma \ref{inv} -- of that rule, and then we apply the inductive hypothesis and the rule, see \cite[Thm. 3.2.2]{NP01} for details. 
 
 If the last step in $\de$ is by a modal or deontic rule, we have two subcases: either (the last step is by one  of $LR$-$C$, $LR$-$R$, $LR$-$K$, $L$-$D^{\P_E}$, $L$-$D^{\P_M}$, $L$-$D^{\P_C}$ and $L$-$D^*$ and) both occurrences of the contraction formula $A$ of $LC$ are principal in the last step or  some  instance of the contraction formula  is    introduced in the appropriate  weakening context of the conclusion. In the first subcase, we apply the inductive hypothesis to the premiss and then the rule. An interesting example is when the last step in $\de$ is by $L$-$D^{\P_E}$. We transform
 
$$
\infer[\infrule LC]{\O B,\g\To\d}{\infer[\infrule L\mbox{-}D^{\P}]{\O B,\O B,\g\To\d}{\deduce{B,B\To\quad}{\vdots\;\de_1}&\deduce{\To B,B}{\vdots\;\de_2}}}
\qquad\textrm{ into }\qquad
\infer[\infrule L\mbox{-}D^{\P}]{\O B,\g\To\d}{\deduce{B\To\qquad}{\vdots\;IH(\de_1)}&\deduce{\To B\qquad}{\vdots\;IH(\de_2)}}
$$
 
\noindent where $IH(\de_1)$ is obtained  by applying the inductive hypothesis for the left rule of contraction to $\de_1$ and $IH(\de_2)$ is obtained by applying the inductive hypothesis for the right rule of contraction to $\de_2$. 

In the second subcase,  we apply an instance of the same   modal or deontic rule which introduces one less occurrence of $A$ in the appropriate context of the conclusion. Let's consider $RC$.  If the last step is by $LR$-$M$ and no  instance of $A$ is principal in the last rule, we transform\vspace{0.3cm}
$$
\infer[\infrule  RC]{\O B,\g'\To\d',A,\O C}{ \infer[\infrule LR\mbox{-}M]{\O B,\g'\To\d',A,A,\O C}{\deduce{B\To C}{\vdots\;\de_1}}}
\qquad\mbox{into}\qquad
\infer[\infrule LR\mbox{-}M]{\O B,\g'\To\d',A,\O C}{\deduce{B\To C\quad}{\vdots\;\de_1}}
$$

\end{proof}
 
 \begin{theorem}\label{cut}The rule of cut is admissible in {\bf G3X}
$$
\infer[\infrule Cut]{\g,\Pi\To\d,\Sigma}{\deduce{\gsd,D}{\vdots\;\de_1}&\deduce{D,\Pi\To\Sigma}{\vdots\;\de_2}}
$$ 
\end{theorem}
 
 \begin{proof}
 We consider an uppermost application of $Cut$ and we show that either it is eliminable, or it can be permuted upward in the derivation until we reach sequents where it is eliminable. The proofs, one for each calculus, are by induction on the weight of the cut formula $D$ with a sub-induction on the sum of the heights of the derivations of the two premisses (cut-height for shortness). The proof can be organized in 3 exhaustive cases:
 
 \begin{enumerate}
 \item At least one of the premisses of cut  is an initial sequent or a conclusion of $L\bot$;
\item The cut formula in not principal in the last step of at least one of the two premisses;
 \item The cut formula is principal in both premisses.
 \end{enumerate}

\medskip
\noindent {\bf{\large\textbullet}$\quad$ Case (1).}$\quad$ Same as for {\bf G3cp}, see \cite[Thm. 3.2.3]{NP01} for the details.
 \medskip
 
 \medskip
\noindent {\bf{\large\textbullet}$\quad$ Case (2).}$\quad$ We have many subcases according to the last rule applied in the derivation  ($\de^\star$)  of the  premiss where the cut formula is not principal. For the propositional rules, we refer the reader to \cite[Thm. 3.2.3]{NP01}, where it is given a procedure that allows to reduce the cut-height. If the last rule applied in $\de^\star$ is a modal or deontic one, we  can transform the derivation into a cut-free one because the conclusion of \mbox{\it{Cut}} is derivable  by replacing the last step of  $\de^\star$ with  the appropriate instance of the same modal or deontic rule. We present explicitly only the cases where the last step of the left premiss is by $LR$-$E$ and $L$-$D^\bot$ and the cut formula is not principal in it, all other transformations being similar.
 \medskip

\noindent $\mathbf{LR\textrm{-}E}:\quad$ \, If the left premiss  is by rule $LR$-$E$  (and $\g\equiv \Box A,\g'$ and $\d\equiv \d',\Box B$), we  transform 

$$
\infer[\infrule Cut]{\O A,\g',\Pi\To\d',\O B,\Sigma}{\infer[\infrule LR\mbox{-}E]{\O A,\g'\To\d',\O B,D}{\deduce{A\To B}{\vdots\;\de_{11}}&\deduce{B\To A}{\vdots\;\de_{12}}}&\deduce{D,\Pi\To\Sigma}{\vdots\;\de_2}}
\qquad\mbox{into}\qquad
\infer[\infrule LR\mbox{-}E ]{\O A,\g',\Pi\To\d',\O B,\Sigma}{\deduce{A\To B}{\vdots\;\de_{11}}&\deduce{B\To A}{\vdots\;\de_{12}}}
$$

\noindent $\mathbf{L}${\bf-}$\mathbf{D^\bot}:\quad$ \, If the left premiss  is by rule $L$-$D^\bot$, we  transform 

$$
\infer[\infrule Cut]{\O A,\g',\Pi\To\d,\Sigma}{\infer[\infrule L\mbox{-}D^\bot]{\O A,\g'\To\d,D}{\deduce{A\To }{\vdots\;\de_{11}}}&\deduce{D,\Pi\To\Sigma}{\vdots\;\de_2}}
\qquad\mbox{into}\qquad
\infer[\infrule L\mbox{-}D^\bot ]{\O A,\g',\Pi\To\d,\Sigma}{\deduce{A\To }{\vdots\;\de_{11}}}
$$

 


\medskip

\noindent {\bf {\large\textbullet}$\quad$ Case (3).}$\quad$ If the cut formula $D$  is principal in both premisses, we have cases according to the principal operator of $D$. In each case we have a procedure that allows to reduce the weight of the cut formula, possibly increasing the cut-height. For the propositional cases, which are the same for all the logics considered here, see \cite[Thm. 3.2.3]{NP01}.

 If $D\equiv\O C$, we consider the different logics one by one, without repeating the common cases. \\

\noindent\textbullet$\quad${\bf G3E(ND).}$\quad$ Both premisses are by rule $LR$-$E$, we have

$$
\infer[\infrule Cut]{\O A,\g',\Pi\To\d,\Sigma',\O B}{\infer[\infrule LR\mbox{-}E]{\O A,\g'\To\d,\O C}{\deduce{A\To C}{\vdots\;\de_{11}}&\deduce{C\To A}{\vdots\;\de_{12}}}&\infer[\infrule LR\mbox{-}E]{\O C,\Pi\To\Sigma',\O B}{\deduce{C\To B}{\vdots\;\de_{21}}&\deduce{B\To C}{\vdots\;\de_{22}}}}
$$

\noindent and we transform it into the following derivation that has two  cuts with cut formulas of lesser weight, which are admissible by IH.

$$
\infer[\infrule LR\mbox{-}E]{\O A,\g',\Pi\To\d,\Sigma',\O B}{\infer[\infrule Cut]{ A\To B}{\deduce{A\To C}{\vdots\;\de_{11}}&\deduce{C\To B}{\vdots\;\de_{21}}}&\infer[\infrule Cut]{B\To A}{\deduce{B\To C}{\vdots\;\de_{22}}&\deduce{C\To A}{\vdots\;\de_{12}}}}
$$

%

\noindent\textbullet$\quad${\bf G3EN(D).}$\quad$ Left premiss by $R$-$N$ and right one by $LR$-$E$. We transform

$$
\infer[\infrule Cut]{\g,\Pi\To\d,\Sigma',\O A}{\infer[\infrule R\mbox{-}N]{\gsd,\O C}{\deduce{\To C}{\vdots\;\de_{11}}}&\infer[\infrule LR\mbox{-}E]{\O C,\Pi\To\Sigma',\O A}{\deduce{C\To A}{\vdots\;\de_{21}}&\deduce{A\To C}{\vdots\;\de_{22}}}}
\qquad\mbox{into}\qquad
\infer[\infrule R\mbox{-}N ]{\g,\Pi\To\d,\Sigma',\O A}{\infer[\infrule Cut]{\To A}{\deduce{\To C}{\vdots\;\de_{11}}&\deduce{C\To A}{\vdots\,\de_{21}}}}
$$

\noindent \textbullet$\quad${\bf G3E(N)D$^\bot$.}$\quad$
Left premiss is by $LR$-$E$, and  right one by $L$-$D^\bot$. We transform 

 $$
\infer[\infrule Cut]{\O A,\g',\Pi\To\d,\Sigma}{\infer[\infrule LR\mbox{-}E]{\O A,\g'\To\d,\O C}{\deduce{A\To C}{\vdots\;\de_{11}}&\deduce{C\To A}{\vdots\;\de_{12}}}&\infer[\infrule L\mbox{-}D^\bot]{\O C,\Pi\To\Sigma}{\deduce{C\To }{\vdots\;\de_{21}}}}
\qquad\mbox{into}\qquad
\infer[\infrule L\mbox{-}D^\bot ]{\O A,\g',\Pi\To\d,\Sigma}{\infer[\infrule Cut]{A\To }{\deduce{A\To C}{\vdots\;\de_{11}}&\deduce{C\To }{\vdots\,\de_{21}}}}
$$

\noindent \textbullet$\quad${\bf G3E(N)D$^\P$.}$\quad$
Left premiss is by $LR$-$E$, and  right one by $L$-$D^{\P_E}$. We transform ($|\Xi|\leq 1$)

{\small $$
\infer[\infrule Cut]{\O A,\g',\O\Xi,\Pi'\To\d,\Sigma}{\infer[\infrule LR\mbox{-}E]{\O A,\g'\To\d,\O C}{\deduce{A\To C}{\vdots\;\de_{11}}&\deduce{C\To A}{\vdots\;\de_{12}}}&\infer[\infrule L\mbox{-}D^{\P_E}]{\O C,\O\Xi,\Pi'\To\Sigma}{\deduce{C,\Xi\To }{\vdots\;\de_{21}}&\deduce{\To C,\Xi}{\vdots\;\de_{22}}}}
\text{into} 
\infer[\infrule L\mbox{-}D^{\P_E}]{\O A,\g',\O\Xi,\Pi'\To\d,\Sigma}{\infer[\infrule Cut]{ \To\Xi,A}{\deduce{\To \Xi, C}{\vdots\;\de_{22}}&\deduce{C\To A}{\vdots\;\de_{12}}}&\infer[\infrule Cut]{A,\Xi\To }{\deduce{A\To C}{\vdots\;\de_{11}}&\deduce{C,\Xi\To }{\vdots\;\de_{21}}}}
$$}

\noindent\textbullet$\quad${\bf G3E(N)D.}$\quad$ Left premiss by $LR$-$E$ and right one by   $L$-$D^\bot$ or $L$-$D^{\P_E}$. Same as  above.\vspace{0.3cm}

\noindent\textbullet$\quad${\bf G3END$^\bot$.}$\quad$ Left premiss by $R$-$N$ and right one by $L$-$D^\bot$. We transform 

 $$
\infer[\infrule Cut]{\g,\Pi\To\d,\Sigma}{\infer[\infrule R\mbox{-}N]{\gsd,\O C}{\deduce{\To C}{\vdots\;\de_{11}}}&\infer[\infrule L\mbox{-}D^\bot]{\O C,\Pi\To\Sigma}{\deduce{C\To }{\vdots\;\de_{21}}}}
\qquad\mbox{into}\qquad
\infer=[\infrule LWs\mbox{ and }RWs ]{\g,\Pi\To\d,\Sigma}{\infer[\infrule Cut]{\phantom{C}\To^{} }{\deduce{\To C}{\vdots\;\de_{11}}&\deduce{C\To }{\vdots\,\de_{21}}}}
$$


\noindent\textbullet$\quad${\bf G3END.}$\quad$ Left premiss by $R$-$N$ and right one by $L$-$D^{\P_E}$. We transform ($|\Xi|\leq 1$)

$$
\infer[\infrule Cut]{\O \Xi,\g,\Pi'\To\d,\Sigma}{
\infer[\infrule R\mbox{-}N]{\g\To\d,\O C}{\deduce{\To C}{\vdots\;\de_{11}}}&
\infer[\infrule L\mbox{-}D^\P]{\O C,\O \Xi,\Pi'\To\Sigma}{\deduce{C,\Xi}{\vdots\;\de_{21}}\To&\To \deduce{C,\Xi}{\vdots\;\de_{22}}}}
\qquad\text{into}\qquad
\infer[\infrule (\star)]{\O \Xi,\g,\Pi'\To\d,\Sigma}{\infer[\infrule Cut]{\Xi\To}{
\deduce{\To C}{\vdots\;\de_{11}}&
\deduce{C,\Xi\To}{\vdots\;\de_{21}}}}
$$
where $(\star)$ is an instance of $L$-$D^\bot$ if $|\Xi|=1$, else ($|\Xi|=0$ and) it is some instances\mbox{ of $LW$ and $RW$.} 
%

\noindent \textbullet$\quad${\bf G3M(ND).}$\quad$
 Both premisses are by rule $LR$-$M$, we transform
 $$
\infer[\infrule Cut]{\O A,\g',\Pi\To\d,\Sigma',\O B}{\infer[\infrule LR\mbox{-}M]{\O A,\g'\To\d,\O C}{\deduce{A\To C}{\vdots\;\de_{11}}}&\infer[\infrule LR\mbox{-}M]{\O C,\Pi\To\Sigma',\O B}{\deduce{C\To B}{\vdots\;\de_{21}}}}
\qquad\mbox{into}\qquad
\infer[\infrule LR\mbox{-}M]{\O A,\g',\Pi\To\d,\Sigma',\O B}{\infer[\infrule Cut]{A\To B }{\deduce{A\To C}{\vdots\;\de_{11}}&\deduce{C\To B}{\vdots\,\de_{21}}}}
$$

%

\noindent\textbullet$\quad${\bf G3MN(D).}$\quad$ Left premiss by $R$-$N$ and right one by $LR$-$M$. Similar to the case with left premiss by $R$-$N$ and right one by $LR$-$E$.\vspace{0.3cm}

\noindent \textbullet$\quad${\bf G3M(N)D$^\bot$} and {\bf G3M(N)D.}$\quad$
Left premiss is by $LR$-$M$, and  right one by  $L$-$D^\bot$ or $L$-$D^{\P_M}$. Similar to the  case with left premiss by $LR$-$E$ and right by $L$-$D^\bot$ or $L$-$D^{\P_E}$, respectively.\vspace{0.3cm}

%

\noindent\textbullet$\quad${\bf G3MND$^\bot$} and {\bf G3MND.}$\quad$ The cases with left premiss by $R$-$N$ and right one by  a deontic rule are like the analogous ones we have already considered.

\medskip
\noindent \textbullet$\quad${\bf G3C(ND).} Both premisses are by rule $LR$-$C$. Let us agree to use $\Lambda$ to denote the non-empty  multiset $A_1,\dots,A_n$, and $\Xi$ for the (possibly empty) multiset $B_2,\dots B_m$.  The derivation

$$
\infer[\infrule Cut]{\O\Lambda,\g',\O\Xi,\Pi'\To\d,\Sigma', \O E}{\infer[\infrule LR\mbox{-}C]{\O\Lambda,\g'\To\d,\O C}{\deduce{\Lambda\To C}{\vdots \;\de_{11}}& \deduce{C\To A_1}{\vdots \;\de_{A_1}}&{}^{\dots}&\deduce{C\To A_n}{\vdots \;\de_{A_n}}}&
\infer[\infrule LR\mbox{-}C]{\O C,\O\Xi,\Pi'\To\Sigma',\O E}{\deduce{C,\Xi\To E}{\vdots \;\de_{21}}&\deduce{ E\To C}{\vdots \;\de_{C}}&{}^{\dots}& \deduce{E\To B_m}{\vdots \;\de_{B_m}}}}
$$

\noindent is transformed into the following derivation having  $n+1$ cuts on formulas of lesser weight
\noindent\scalebox{0.8000}{$$
\infer[\infrule LR\mbox{-}C]{\O\Lambda,\g',\O\Xi,\Pi'\To\d,\Sigma', \O E}{\infer[\infrule Cut]{\Lambda,\Xi\To E}{\deduce{\Lambda\To C}{\vdots\:\de_{11}}&\deduce{C,\Xi\To E}{\vdots\;\de_{21}}}&
\infer[\infrule{Cut}\;\dots]{E\To A_1}{\deduce{E\To C}{\vdots\:\de_{C}}&\deduce{C\To A_1}{\vdots\;\de_{A_n}}}&
\infer[\infrule Cut]{E\To A_n}{\deduce{E\To C}{\vdots\:\de_{C}}&\deduce{C\To A_n}{\vdots\;\de_{A_n}}}&
 \deduce{E\To B_1}{\vdots \;\de_{B_1}}
 &{}^{\dots}& \deduce{E\To B_m}{\vdots \;\de_{B_m}}
 }
$$}\vspace{0.3cm}

\noindent \textbullet$\quad${\bf G3CN(D).} Left premiss by $R$-$N$ and right premiss by $LR$-$C$. We have

$$
\infer[\infrule Cut]{\g,\O A_1,\dots,\O A_n,\Pi'\To\d,\Sigma', \O B}{\infer[\infrule R\mbox{-}N]{\gsd,\O C}{\deduce{\To C}{\vdots\;\de_{11}}}&
\infer[\infrule LR\mbox{-}C]{\O C,\O A_1,\dots,\O A_n,\Pi'\To\Sigma',\O B}{\deduce{C,A_1,\dots,A_n\To B}{\vdots \;\de_{21}}&\deduce{ B\To C}{\vdots \;\de_{C}}&{}^{\dots}& \deduce{B\To A_n}{\vdots \;\de_{A_n}}}}
$$
 where $A_1,\dots,A_n$ (and thus also $\O A_1,\dots,\O A_n$)  may or may not be the empty multiset. If $A_1,\dots,A_n$ is not empty, we transform it  into the following derivation having one cut with cut formula of lesser weigh
 
$$
\infer[\infrule LR\mbox{-}C]{\g,\O A_1,\dots\O A_n,\Pi'\To\d,\Sigma', \O B}{
\infer[\infrule Cut]{A_1,\dots,A_n\To B}{\deduce{\To C\quad}{\vdots\:\de_{11}}&\deduce{C,A_1,\dots, A_n\To B}{\vdots\;\de_{21}}}&
 \deduce{B\To A_1}{\vdots \;\de_{A_1}}
 &{}^{\dots}& \deduce{B\To A_n}{\vdots \;\de_{A_n}}
 }
$$
If, instead, $A_1,\dots,A_n$ is empty, we transform it into

$$
\infer[\infrule R\mbox{-}N]{\g,\Pi'\To\d,\Sigma',\O B}{\infer[\infrule Cut]{\To B}{\deduce{\To C\qquad}{\vdots\:\de_{11}}&\deduce{C\To B}{\vdots\;\de_{21}}}}
$$

\noindent \textbullet$\quad${\bf G3CD$^{\P}$.} Left premiss by $LR$-$C$ and right premiss by $L$-$D^{\P_C}$. We transform  (we assume $\Xi= A_1,\dots A_k$, $\Theta= C, B_2,\dots, B_m$ and $\Lambda= D_1,\dots, D_n$) 

$$
\infer[\infrule Cut ]{\O\Xi,\O B_1,\dots, \O B_m,\O \Lambda,\g',\Pi'\To\d,\Sigma}{
\infer[LR\mbox{-}C]{\O\Xi,\g'\To\d,\O C}{\deduce{\Xi\To C}{\vdots\;\de_{11}}&\deduce{\{ C\To A_i\,|\, A_i\in\Xi\}}{\vdots\;\de_{1A_i}}}&
\infer[\infrule L\mbox{-}D^{\P_C}]{\O C,\O B_2,\dots,\O B_m,\O\Lambda,\Pi'\To\Sigma}{
\deduce{\Theta,\Lambda\To}{\vdots\;\de_{21}}&
\deduce{\{ \To E,D_j\,|\, E\in\Theta\text{ and } D_j\in\ \Lambda\}}{\vdots\;\de_{\Theta_{i}\Lambda_j}}}}
$$
into the following derivation having $1+(k\times n)$ cuts on formulas of lesser weight

$$\scalebox{0.830}{
\infer[\infrule{ L\mbox{-}D^{\P_C}}]{\O\Xi,\O B_1,\dots, \O B_m,\O \Lambda,\g',\Pi'\To\d,\Sigma}{
\infer[\infrule Cut]{\Xi,B_2,\dots, B_m,\Lambda\To}{
\deduce{\Xi\To C}{\vdots\;\de_{11}}&\deduce{C,B_2,\dots,B_m,\Lambda\To}{\vdots\;\de_{21}}}
&
\infer[\infrule Cut]{\{\To A_i,D_j| A_i\in\Xi,\, D_j\in\Lambda\}}{\deduce{\To D_j,C}{\vdots\;\de_{\Theta_1,\Lambda_j}}&\deduce{C\To A_i}{\vdots\;\de_{1A_i}}}
&
\deduce{\{\To B_i,D_j| B_i\in\Theta-C,\, D_j\in \Lambda\}}{\vdots\;\de_{\Theta_i\Lambda_j}}
}
}$$
\noindent \textbullet$\quad${\bf G3C(N)D.}$\quad$ Left premiss by $LR$-$C$ and right one by $L$-$D^*$. It is straightforward to transform the derivation into another one having one cut with cut formula of lesser weight. \medskip

\noindent \textbullet$\quad${\bf G3R(D).}$\quad$ Both premisses are by rule $LR$-$R$, we transform

\noindent\scalebox{0.88000}{ $$
\infer[\infrule Cut]{\O A,\O\Xi,\g',\O\Psi,\Pi'\To\d,\Sigma',\O B}{\infer[\infrule LR\mbox{-}R]{\O A,\O\Xi,\g'\To\d,\O C}{\deduce{A,\Xi\To C}{\vdots\;\de_{11}}}&\infer[\infrule LR\mbox{-}R]{\O C,\O\Psi,\Pi'\To\Sigma',\O B}{\deduce{C,\Psi\To B}{\vdots\;\de_{21}}}}
\quad\mbox{into}\quad
\infer[\infrule LR\mbox{-}R]{\O A,\O\Xi,\g',\O\Psi,\Pi'\To\d,\Sigma',\O B}{\infer[\infrule Cut]{A,\Xi,\Psi\To B }{\deduce{A,\Xi\To C}{\vdots\;\de_{11}}&\deduce{C,\Psi\To B}{\vdots\,\de_{21}}}}
$$}
\medskip

\noindent\noindent \textbullet$\quad${\bf G3RD$^\star$.}$\quad$Left premiss is by $LR$-$R$, and  right one by $L$-$D^\star$, we transform 

\noindent$$
\infer[\infrule Cut]{\O A,\O\Xi,\g',\O\Psi,\Pi'\To\d,\Sigma}{\infer[\infrule LR\mbox{-}R]{\O A,\O\Xi,\g'\To\d,\O C}{\deduce{A,\Xi\To C}{\vdots\;\de_{11}}}&\infer[\infrule L\mbox{-}D^*]{\O C,\O\Psi,\Pi'\To\Sigma}{\deduce{C,\Psi\To }{\vdots\;\de_{21}}}}
\quad\mbox{into}\quad
\infer[\infrule L\mbox{-}D^*]{\O A,\O\Xi,\g',\O\Psi,\Pi'\To\d,\Sigma}{\infer[\infrule Cut]{A,\Xi,\Psi\To  }{\deduce{A,\Xi\To C}{\vdots\;\de_{11}}&\deduce{C,\Psi\To }{\vdots\,\de_{21}}}}
$$ 


\noindent \textbullet$\quad${\bf G3K(D).}$\quad$
The new cases with respect to {\bf G3R(D)} are those with left premiss by an instance of $LR$-$K$ that has no principal formula in the antecedent. These cases can be treated like cases with left premiss by $R$-$N$.\end{proof}

%
%

%
\section{Decidability and syntactic completeness}\label{secdecax}
\subsection{Decision procedure for {\bf G3X}}\label{decision}
Each calculus {\bf G3X} has the strong subformula property since all  active formulas of  each rule  in Tables \ref{G3cp} and \ref{Modal rules} are  proper subformulas of the the principal formulas and no formula  disappears in moving from premiss(es) to conclusion. As usual, this gives us a syntactic proof of consistency.
\begin{proposition}\label{cons}\
\begin{enumerate}
\item Each  premiss of each rule of {\bf G3X} has smaller weight than its conclusion;
\item Each premiss of each modal or deontic rule of {\bf G3X} has smaller modal depth than its conclusion;
\item The calculus {\bf G3X} has the subformula property:  a {\bf G3X}-derivation of a sequent $\mathcal{S}$ contains only sequents composed of subformulas of $\mathcal{S}$;
\item The empty sequent is not {\bf G3X}-derivable.
\end{enumerate}\end{proposition} 
We also have an effective method to decide the derivability of a sequent in {\bf G3X}: we start from the desired sequent $\g\To\d$ and we construct all possible {\bf G3X}-derivation trees until either we find a tree where each leaf is an initial sequent or a conclusion of $L\bot$ -- we have found a {\bf G3X}-derivation of $\g\To\d$ -- or we have checked all possible {\bf G3X}-derivations   and we have found none --  $\g\To\d$ is not {\bf G3X}-derivable.  

More in details, we present  a depth-first procedure that tests {\bf G3X}-derivability in polynomial space. As it is usual in decision procedures involving non-invertible rules, we have trees involving two kinds of branching. Application of a rule with more than one premiss produce an \emph{AND-branching} point, where all branches have to be derivable to obtain a derivation. Application of a non-invertible rule to a sequent that can be the conclusion of different instances of non-invertible rules produces an \emph{OR-branching} point, where only one branch need be derivable to obtain a derivation.
In the procedure below we will assume that, given a calculus {\bf G3X} and given a sequent $\O\Pi,\g^p\To\d^p,\O\Sigma$ (where $\g^p$ and $\d^p$ are multisets of propositional variables), there is some fixed way of ordering the finite (see below) set of instances of modal and deontic rules of {\bf G3X} (\emph{{\bf X}-instances}, for shortness) having that sequent as conclusion. Moreover, we will represent the root of branches above an OR-branching point  by nodes of shape $\Box_i$, where $\Box_i$ is the name of the $i$-th {\bf X}-instance applied (in the order of  all {\bf X}-instances having that conclusion).  To illustrate, if we are in {\bf G3EN} and we have to consider  $\O A,\O B,\g^p\To\d^p,\O C$ then we obtain (fixing one way of ordering the three {\bf X}-instances having that sequent as conclusion):
 
$$
\infer{\O A,\O B,\g^p\To\d^p,\O C}{
\infer[\qquad]{LR\mbox{-}E}{A\To C& C\To A}&
\infer[\qquad]{LR\mbox{-}E}{B\To C&C\To B}&
\infer{R\mbox{-}N}{\To B}}
$$
where the lowermost sequent is an  OR-branching point  and the two nodes   $LR$-$E_1$ and $LR$-$E_2$ are AND-branching points. Finally, Given an AND(OR)-branching point 

$$\infer{\mathcal{S}}{\mathcal{S}_1&\dots&\mathcal{S}_n}$$ we say that the branch above $\mathcal{S}_i$ is an  \emph{unexplored AND(OR)-branch} if no one  of its nodes has already been active. 

\begin{definition}[{\bf G3X}-decision procedure]\label{decisiontree}\
\begin{description}
\item[Stage 1.] We write the one node tree $
\g\To\d$
and we label $\g\To\d$ as active.
\item[Stage n+1.] Let $\mathcal{T}_n$ be the tree constructed at stage $n$, let $\mathcal{S}\equiv\Pi\To\Sigma$ be its active sequent, and let $\mathcal{B}$ be the branch  going from   the root  of $\mathcal{T}_n$ to $\mathcal{S}$. 
\begin{description}
\item[Closed.] If $\mathcal{S}$ is such that $p\in\Pi\cap\Sigma$ (for some propositional variable $p$) or $\bot\in\Pi$, then we label $\mathcal{S}$ as closed and
\begin{description}
\item[Derivable.] If $\mathcal{B}$ contains no  unexplored AND-branch, the procedure ends and  \mbox{$\g\To\d$} is {\bf G3X}-derivable;
\item[AND-backtrack.]If, instead, $\mathcal{B}$ contains  unexplored AND-branches, we choose the topmost one and we label as active its leftmost unexplored leaf. {\bf Else}
\end{description}
\item[Propositional.] if $\mathcal{S}$ can be the conclusion of some instances $\circ_1,\dots,\circ_m$ of the invertible propositional rules, we extend  $\mathcal{B}$ by applying  one of such instances: 

$$\infer[\infrule{\circ_i\quad 1\leq i\leq m}]{\mathcal{S}}{\mathcal{S}_1&(\mathcal{S}_2)}$$ \noindent where, if  $\mathcal{S}_2$, if present,  $\mathcal{S}$ is an AND-branching point. {\bf Else}
\item[Modal.] If $\mathcal{S}$ can be the conclusion of the following  canonically ordered list of   {\bf X}-instances:

$$
\infer[\infrule \Box_1]{\mathcal{S}}{\mathcal{S}_1^1&\dots&\mathcal{S}^1_k }
\qquad\deduce[\dots]{\phantom{a}}{}\qquad
\infer[\infrule \Box_m]{\mathcal{S}}{\mathcal{S}_1^m&\dots&\mathcal{S}^m_l }
$$ 
then we extend  $\mathcal{B}$ as follows:

$$
\infer{\mathcal{S}}{\infer{\Box_1\phantom{^1}}{\mathcal{S}_1^1&\dots&\mathcal{S}^1_k }&\deduce[\dots]{\phantom{a}}{}&\infer{\Box_m\phantom{^1}}{\mathcal{S}_1^m&\dots&\mathcal{S}^m_l}
}
$$
where, if $m\geq 2$,  $\mathcal{S}$ is  OR-branching  and, if $\Box_i$ is a rule with more than one premiss, $\Box_i$ is AND-branching. Moreover, we label $\mathcal{S}_1^1$ as active. {\bf Else} 
\item[Open.] No rule of {\bf G3x} can be applied to $\mathcal{S}$, then  we label $\mathcal{S}$ as open and
\begin{description}
\item[Underivable.] If $\mathcal{B}$ contains no unexplored  OR-branch, the procedure ends and  \mbox{$\g\To\d$} is not  {\bf G3X}-derivable;
\item[OR-backtrack.]If, instead, $\mathcal{B}$ contains unexplored OR-branches, we choose the topmost one and we label as active its leftmost unexplored leaf. 
\end{description}
\end{description}
\end{description}
\end{definition}
 
Termination  can be shown as follows. Proposition \ref{cons}.1 entails that the height of each branch of the tree $\mathcal{T}$  constructed in a {\bf G3X}-decision procedure for a sequent $\g\To\d$  is bounded by the weight of  $\g\To\d$  (in particular, given Proposition \ref{cons}.2, the number of OR-branching points occurring in a branch is bounded by the modal depth of $\g\To\d$). Moreover, $\mathcal{T}$ is  finitary branching  since all rules of {\bf G3X} are finitary branching rules, and since each sequent can be the conclusion of a finite number  $k$ of  {\bf X}-instances (for each {\bf G3X} $k$ is bounded by a function of $|\g|$ and $|\d|$)}.   Hence, after a finite number of stages we are either in case {\bf Derivable} or in case {\bf Underivable} and, in both cases, the procedure ends. In the first case we can easily extract a {\bf G3X}-derivation of $\g\To\d$ from $\mathcal{T}$ (we just have to delete  all unexplored branches as well as all underivable sub-trees   above an OR-branching point). In the latter case, thanks to Proposition \ref{cons}.3, we know that (modulo the order of the invertible propositional rules)  we have explored the whole search space for a {\bf G3X}-derivation of $\g\To\d$ and we have found none. 


We  prove that it is possible to test {\bf G3X}-derivability in polynomial space by showing how it is possible to store only the active node together with a stack containing  information sufficient to reconstruct unexplored branches. For the propositional part of the calculi, we proceed as in \cite{B97,H93,H95}: each entry of the stack is a triple containing the name of the rule applied, an index recording which of its premisses  is active, and its principal formula. For the  {\bf X}-instances two complications arise: we need to record which OR-branches are unexplored yet, and we have to keep track of the weakening contexts of the conclusion in the premisses of {\bf X}-instances. The first problem has already been  solved by having assumed that the {\bf X}-instances applicable to a given sequent have a fixed canonical order.  The second problem is solved by adding a numerical superscript to the formulas occurring in a sequent and by imposing that:\\
- All formulas in the end-sequent have 1 as superscript;\\
 - The superscript $k$ of the principal formulas of rules and of initial sequents are maximal in that sequent;\\- Active formulas of {\bf X}-instances (propositional rules) have $k+2$ ($k$, respectively) as superscript;\\-  
  Contexts are copied in the premisses of each rule.\\ By doing so, the contexts of the conclusion are copied in the premisses in each rule of {\bf G3X},  but they cannot be principal in the trees above the premisses of the  {\bf X}-instances because their superscript is never maximal therein.  It is immediate to see the the superscripts occurring in a derivation are bounded by (twice) the modal depth of the end-sequent.
  
  Instances of all modal and deontic rules in Table \ref{Modal rules} but $LR$-$C$ and $L$-$D^{\P_C}$ are such that there is no need to  record their principal formulas in the stack entry: they are the boxed version of the formulas  having maximal superscript in the active premiss; moreover, the name of the rule and the number of the premiss  allow to reconstruct the position of the principal formulas (for the right premiss of  $LR$-$E$ and $L$-$D^{\P_E}$, we have to switch the two formulas). In instances of rules $LR$-$C$ and $L$-$D^{\P_C}$, instead,  this  doesn't hold since in all premisses but the leftmost one  there is no subformula of some principal formulas. We can overcome this problem by copying in each premiss all principal formulas having no  active subformula in that premiss and by adding one to  their superscript. We also keep fixed the position of all formulas (modulo the swapping of the two active formulas). To illustrate, one such instance is:
  
  $$ \scalebox{0.9}{\infer[\infrule LR\mbox{-}C]{\Box A^k_1,\Box A^k_2,\g\To\d,\O B^k}{A^{k+2}_1,A^{k+2}_2,\g\To\d, B^{k+2}\quad&
 B^{k+2},\Box A^{k+1}_2,\g\To\d, A_1^{k+2}
 \quad&
 \Box A^{k+1}_1,B^{k+2},\g\To\d, A_2^{k+2}
 }
}$$
 In this way, given  the name of the modal or deontic rule applied, any premiss of this rule instance, and its position among the premisses of this rule, we can reconstruct both the conclusion of this rule instance and  its position  in the fixed order of {\bf X}-instances concluding that sequent (thus we know which OR-branches are unexplored yet). In doing so, we use the hp-admissibility of contraction to ensure that no formula has more than one occurrence in the antecedent or in the succedent of the conclusion of {\bf X}-instances (otherwise we might be unable to reconstruct which of two identical {\bf X}-instances we are considering). Hence, for {\bf X}-instances each stack entry records the name of the rule applied and an index recording which premiss we are considering.

The decision procedure is like in Definition \ref{decisiontree}. The only novelty is that at each stage, instead of storing the full tree constructed so far, we store only the active node and the stack, we push an entry  in the stack and, if we are in a backtracking case, we pop stack entries (and we  use them  to reconstruct the corresponding active sequent) until we reach an entry recording unexplored branches of the appropriate kind, if any occurs. 

\begin{theorem}
{\bf G3X}-derivability is decidable in $\mathcal{O}(n\, \log{} n)$-{\sc space}, where $n$ is the weight of the end-sequent.
\end{theorem}

\begin{proof}
We have already argued that proof search terminates.
As   in \cite{B97,H93,H95}, Proposition \ref{cons}.1 entails that the  stack depth is bounded by $\mathcal{O}(n)$ and, by storing the principal formulas of propositional rules as indexes into the end-sequent,  each entry requires $\mathcal{O}(\log{}n)$ space. Hence we have an $\mathcal{O}(n\, \log{} n)$  space bound for the stack. Moreover, the active sequent contains at most $\mathcal{O}(n)$ subformulas of the end-sequent and their numerical superscripts. Each such subformula requires $\mathcal{O}(\log{}n)$ space since it can be recorded as an index into the end-sequent;  its numerical superscript requires $\mathcal{O}(\log{}n)$ too since there are at most $\mathcal{O}(n)$ superscripts. Hence also the active sequent requires $\mathcal{O}(n\log{}n)$ space.
\end{proof}
%

\subsection{Equivalence with the axiomatic systems}
It is now time to show that the sequent calculi introduced are equivalent to the non-normal logics of Section \ref{secaxiom}.  We write {\bf G3X} $\vdash \g\To\d$ if the sequent $\g\To\d$ is  derivable in {\bf G3X}, and we say that $A$ is derivable in {\bf G3X} whenever {\bf G3X} $\vdash \;\To A$. We begin by proving the following 

\begin{lemma}\label{ax} All the axioms of the  axiomatic system {\bf X} are derivable in {\bf G3X}.\end{lemma}

\begin{proof} A straightforward application of the rules of the appropriate sequent calculus, possibly using Proposition \ref{genax}.  As an example, we show that the deontic axiom $D^\bot$ is derivable by means of rule $L$-$D^\bot$ and that axiom  $C$ is derivable by means of  $LR$-$C$.

$$
\infer[\infrule R\neg]{\To\neg\O\bot}{\infer[\infrule L\mbox{-}D^\bot]{\O\bot\To}{\infer[\infrule L\bot]{\bot\To}{}}}
\qquad
\infer[\infrule R\to]{\To \O A\wedge\O B\to\O(A\wedge B)}{\infer[\infrule L\wedge]{\O A\wedge \O B\To \O (A\wedge B)}{\infer[\infrule LR\mbox{-}C]{\O A,\O B\To \O(A\wedge B)}{\infer[\infrule R\wedge]{A,B\To A\wedge B}{\infer[\infrule{\ref{genax}}]{A,B\To A}{}&\infer[\infrule{\ref{genax}}]{A,B\To B}{}}&
\infer[\infrule L\wedge  ]{A\wedge B\To A}{\infer[\infrule{\ref{genax}}]{A,B\To A}{}}&
\infer[\infrule L\wedge]{A\wedge B\To B}{\infer[\infrule{\ref{genax}}]{A,B\To B}{}}}}}
$$

\end{proof}

Next we prove  the equivalence of the sequent calculi for non-normal logics with the corresponding axiomatic systems in the sense that  a sequent $\g\To\d$ is derivable in {\bf G3X} if and only if  its characteristic formula $\bigwedge\g\to\bigvee \d$ is derivable in {\bf X} (where the empty antecedent stands for $\top$ and the empty succedent for $\bot$).  As a consequence  each calculus is sound and complete with respect to the appropriate class of neighbourhood models (see  Section \ref{semantics}).

\begin{theorem}\label{comp} Derivability in  the sequent system {\bf G3X} and in the axiomatic system {\bf X} are equivalent, i.e.

\begin{center}
{\bf G3X} $\vdash\;\g\To\d \qquad$iff$\qquad${\bf X} $\vdash\bigwedge\g\to\bigvee\d$
\end{center}\end{theorem}

\begin{proof}
To prove the right-to-left implication, we argue by induction on the height of the  axiomatic derivation in  {\bf X}. The base case is covered by Lemma \ref{ax}. For the inductive steps, the case of $MP$ follows by the admissibility of Cut and the invertibility of rule $R\to$.  If the last step is by $RE$, then $\Gamma=\emptyset$ and  $\Delta$ is $\O C\leftrightarrow \O D$. We know that (in {\bf X}) we have derived $\O C\leftrightarrow\O D$ from $C\leftrightarrow D$. Remember that $C\leftrightarrow D$ is defined as $(C\to D)\wedge (D\to C)$.  Thus we assume, by inductive hypothesis (IH) , that {\bf G3ED} $\vdash\;  \To C\to D\wedge D\to C$. From this, by invertibility of  $R\wedge$ and $R\to$ (Lemma \ref{inv}), we obtain that {\bf G3ED} $\vdash\;  C\To D$ and {\bf G3ED} $\vdash\;  D\To C$.    We can thus  proceed as follows \vspace{0.3cm}

$$
\infer[\infrule R\wedge]{\To (\O C\to \O D)\wedge(\O D\to \O C)}{\infer[\infrule R\to]{\To \O C\to\O D}{\infer[\infrule LR\mbox{-}E]{\O C\To \O D}{\infer[IH+\ref{inv}]{C\To D}{}&\infer[IH+\ref{inv}]{D\To C}{}}
}&
\infer[\infrule R\to]{\To \O D\to\O C}{\infer[\infrule LR\mbox{-}E]{\O D\To \O C}{\infer[IH+\ref{inv}]{D\To C}{}&\infer[IH+\ref{inv}]{C\To D}{}}}}
$$


For the converse implication, we assume {\bf G3X} $\vdash \g\To\d$, and show, by induction on the height of the derivation in sequent calculus, that {\bf X} \mbox{$\vdash \bigwedge \g\to\bigvee\d$.} If the derivation has height 0,  we have an initial sequent -- so $\g\cap\d\neq\emptyset$ -- or an instance on $L\bot$ -- thus $\bot\in\g$. In both cases the claim holds. If the height is $n+1$, we consider the last rule applied in the derivation. If it is a propositional one, the proof is straightforward. If it is a modal rule, we argue by cases.

If the last step of a derivation in {\bf G3E(ND)} is by  $LR$-$E$, we have derived $\O C,\g'\To\d',\O D$ from $C\To D$ and $D\To C$. By IH and propositional reasoning, {\bf ED} $\vdash C\leftrightarrow D$, thus {\bf ED} $\vdash \O C\to \O D$. By some propositional steps we conclude {\bf ED} $\vdash (\O C\wedge\bigwedge\g')\to (\bigvee\d'\lor\O D).$  The cases of $LR$-$M$,  $LR$-$R$, and $LR$-$K$  can be treated in a similar manner (thanks, respectively, to the rule $RM$, $RR$,  and $RK$ from Table \ref{rulesinf}).

If we are in {\bf G3C(ND)}, suppose the last step is the following instance of $LR$-$C$:
$$
\infer[\infrule LR\mbox{-}C]{\O C_1,\dots,\O C_k,\g'\To \d',\O D}{C_1,\dots C_k\To D&D\To C_1&\dots& D\To C_k}
$$
By IH, we have  that  {\bf C(ND)} $\vdash D\to C_i$ for all $i\leq k$, and, by propositional reasoning, we have that  {\bf C(ND)} $\vdash D\to C_1\wedge\dots\wedge C_k$. We also know, by IH, that  {\bf C(ND)} $\vdash C_1\wedge\dots\wedge C_k\to D$. By applying $RE$ to these two theorems we get that 
\begin{equation}\label{3}
\mathbf{ C(ND)} \vdash \O(C_1\wedge\dots\wedge C_k)\to \O D
\end{equation}
 By using axiom $C$ and propositional reasoning, we know that 
 \begin{equation}\label{4}
 \mathbf{ C(ND)} \vdash \O C_1\wedge\dots\wedge\O C_k\to\O(C_1\wedge\dots\wedge C_k)
 \end{equation}
 By applying transitivity to (\ref{4}) and (\ref{3}) and some propositional steps, we conclude that 
 $$
  \mathbf{ C(ND)} \vdash (\O C_1\wedge\dots\wedge\O C_k\wedge \bigwedge\g')\to (\bigvee\d'\lor\O D)
  $$
  
  Let's now consider rule $L$-$D^\bot$.  Suppose we are in {\bf G3XD$^\bot$}  and  we have derived  $\O C,\g'\To\d$ from $C\To$. By IH,  {\bf XD$^\bot$} $\vdash C\to\bot$, and we know that {\bf xD$^\bot$} $\vdash \bot \to C$. Thus by $RE$ (or $RM$), we get {\bf XD$^\bot$} $\vdash \O C\to \O \bot$. By  contraposing it and then applying a $MP$ with the axiom $D^\bot$, we get that {\bf XD$^\bot$} $\vdash \neg\O C$. By some easy propositional steps we  conclude  {\bf XD$^\bot$} $\vdash ( \O C\wedge\bigwedge\g')\to \bigvee\d$. The case $R$-$N$ is similar. 

Let's consider rules $L$-$D^{\P_E}$. Suppose we are in {\bf G3ED$^\P$}  and we have derived \mbox{$\O A,\O B,\g'\To\d$} from the premisses $A,B\To$ and $\To A,B$. By induction we get that {\bf ED$^\P$}$\vdash A\wedge B\to\bot$ and {\bf ED$^\P$}$\vdash A\lor B$. Hence, {\bf ED$^\P$}$\vdash B\to \neg A$ and {\bf ED$^\P$}$\vdash\neg A\to B$. By applying $RE$ we get that  

$$ \mathbf{ED^\P}\vdash \Box B\to\Box \neg A$$
 which, thanks to  axiom $D^\P$, entails that  
 
 $$\mathbf{ ED^\P}\vdash \Box B\to\neg\Box A$$ 
  By some propositional steps we conclude   
  
 $$\mathbf{ ED^\P}\vdash (\Box A\wedge\Box B\wedge \bigwedge\g')\to\bigvee\d$$
   Notice that, thanks to Proposition \ref{cons}.4 and Theorem \ref{cut},  we can assume that instances of rule $L$-$D^\P$ always  have   two principal formulas. Otherwise the calculus would prove the empty sequent (we will also assume that neither $\Pi$ nor $\Sigma$ is empty in instances of rule $L$-$D^{\P_C}$).
   
    The case of $L$-$D^{\P_M}$ is analogous to that of $L$-$D^\bot$ for instances with one principal formula and to that of $L$-$D^{\P_E}$ for instances with two principal formulas.

Let's consider rule $L$-$D^{\P_C}$. Suppose we have a {\bf G3CD$^\P$}-derivation whose last step is:
$$
\infer{\O\Pi,\O\Sigma,\g'\To \d'}{\Pi,\Sigma\To &\{\To A,B| \,A\in\Pi\text{ and }B\in \Sigma\}}
$$
By induction and by some easy propositional steps we   know that {\bf ECD$^\P$} $\vdash \bigwedge\Pi\leftrightarrow\neg\bigwedge\Sigma$. By rule $RE$ we derive {\bf ECD$^\P$} $\vdash\O\bigwedge\Pi\to\O\neg\bigwedge\Sigma$, which, thanks to axiom $D^\P$, entails that {\bf ECD$^\P$} $\vdash\O\bigwedge\Pi\to\neg\O\bigwedge\Sigma$. By transitivity with two (generalized) instances of axiom $C$ we obtain {\bf ECD$^\P$} $\vdash \bigwedge\O\Pi\to \neg\bigwedge\O\Sigma$. By some easy  propositional steps we conclude that  \mbox{{\bf ECD$^\P$} $\vdash (\bigwedge\O\Pi\wedge\bigwedge\O\Sigma\wedge\bigwedge\g')\to\bigvee\d$.}

The admissibility of $L$-$D^*$ in {\bf EC(N)D}, {\bf RD}, and {\bf KD} is similar to that of $LR$-$C$:  in (\ref{3}) we replace $\O D$ with $\O \bot$ and then we use theorem $D^\bot$ to transform it into $\bot$.
\end{proof}
By combining this and Theorem \ref{compax} we have the following result.
\begin{corollary}
The calculus {\bf G3X} is sound and complete with respect to the class of all neighbourhood models for {\bf X}.
\end{corollary}

%

\subsection{Forrester's Paradox}\label{forrester}
  As an application of our decision procedure, we use it to analyse two formal reconstructions of Forrester's paradox \cite{F84}, which   is one of the many paradoxes that endanger the normal deontic logic {\bf KD} \cite{M06}. Forrester's informal argument goes as follows:

\begin{quote}

Consider the following three statements:
\begin{enumerate}
\item Jones murders Smith.
\item Jones ought not murder Smith.
\item  If Jones murders Smith, then Jones ought to murder Smith gently.
\end{enumerate}
Intuitively, these sentences appear to be consistent. However 1 and 3 together imply that
\begin{itemize}
\item[4.] Jones ought to murder Smith gently.
\end{itemize}
Also we accept the following conditional:
\begin{itemize}
\item[5.]  If Jones murders Smith gently, then Jones  murder Smith.
\end{itemize}
Of course, this is \emph{not} a logical validity but, rather, a fact about the world we live in. Now, if we assume that the monotonicity rule is valid, then statement 5 entails
\begin{itemize}
\item[6.]  If Jones ought to murder Smith gently, then Jones  ought to murder Smith.
\end{itemize}
And so, statements 4 and 6 together imply
\begin{itemize}
\item[7.] Jones ought to murder Smith.
\end{itemize}
But [given the validity of $D^\P$] this contradicts statement 2. The above argument suggests  that classical deontic logic should \emph{not} validate the monotonicity rule [$RM$] \cite[p. 16]{P17}  
\end{quote}

 We  show that Forrester's paradox is not a valid argument in deontic logics by presenting, in Figure \ref{fig}, a failed {\bf G3KD}-proof search  of the sequent that expresses it:
\begin{equation}
g\to m,m\to\Box g,\O\neg m,m\To
\end{equation}
 where $m$ stands for 'John \underline{m}urders Smith' and $g$ for `John murders Smith \underline{g}ently' \cite[pp. 87--91]{M06}.
Note that, by Theorem \ref{comp}, if Forrester's paradox  is not {\bf G3KD}-derivable, then it is not valid  in all the weaker  deontic logics we have considered.

\begin{figure}
$$\infer[\infrule L\to]{g\to m,m\to\Box g,\O\neg m,m\To}{
\infer[\infrule L\to]{m\to\Box g,\O\neg m,m\To g}{
\deduce{m,\O \neg m\To g,m^{\phantom{a}}}{\textnormal{closed}}&
\infer{\O g,\O\neg m,m\To g^{\phantom{a}}}{
\infer{L\mbox{-}D^\star}{
\infer[\infrule L\neg]{g,\neg m\To}{\deduce{g\To m}{\textnormal{open}}}}
&
\infer{L\mbox{-}D^\star}{\deduce{g\To }{\textnormal{open}}}
&
\infer{L\mbox{-}D^\star}{\infer[\infrule L\neg]{\neg m\To}{\deduce{\To m}{\textnormal{open}}}}}
}&
\deduce{m,m\to\Box g,\O\neg m,m\To^{\phantom{A}}}{\vdots}}$$
\caption{Failed {\bf G3KD}-proof search of Forrester's paradox \cite{M06}}\label{fig}
\end{figure}

 To make our failed proof search into a derivation of Forrester's paradox, we would have to add (to {\bf G3MD$^\P$} or stronger calculi) a non-logical axiom $\To g\to m$, and to have cut as a primitive -- and  ineliminable -- rule of inference.  An  Hilbert-style axiomatization of Forrester's argument -- e.g., \cite[p. 88]{M06} -- hides this cut with a non-logical axiom in the  step where   $\O g\to\O m$ is derived from $ g\to m$,  by one of $RM$, $R R$ or $RK$. This step -- i.e., the step from 5 to 6 in the informal argument above -- is not acceptable because none of these rules allows to infer its conclusion when the premiss is  an assumption  and not a  theorem. We have here an instance of the same  problem that has led many authors to conclude that the deduction theorem fails in modal logics, conclusion that has been  shown to be wrong in \cite{NH12}. 
 
 An alternative formulation of Forrester's argument is given in  \cite{T97}, where the sentence `John murders Smith gently' is expressed by the complex formula $g\wedge m$ instead of by the atomic $g$. In this case Forrester's argument becomes valid whenever the monotonicity rule is valid as it shown in Figure \ref{figfig}. Nevertheless, whereas it was an essential ingredient of the informal version, under this formalization premiss 5 becomes dispensable. Hence it is disputable that this is an acceptable way of formalising Forrester's argument.
 \begin{figure}
$$ \infer[\infrule L\to]{m\to\O(m\wedge g),\O\neg m,m\To}{
 \deduce{\O\neg m,m\To m^{\phantom{a}}}{\textnormal{closed}}&
\infer{\O(g\wedge m),\O\neg m,m\To}{ \infer{L\mbox{-}D^\star}{
\infer[\infrule L\wedge]{g\wedge m,\neg m\To}{
\infer[\infrule L\neg]{g, m,\neg m\To}{
\deduce{g,m\To m}{\textnormal{closed}}}}}&
  \infer{L\mbox{-}D^\star}{g\wedge m\To}&
   \infer{L\mbox{-}D^\star}{\neg m\To}}
  }$$
 \caption{ Succesfull {\bf G3MD}-proof search for the alternative version of  Forrester's paradox \cite{T97}}\label{figfig}
 \end{figure}
  
 This is not the place to discuss at length  the correctness of formal representation of Forrester's argument  and their implications for deontic logics. We just wanted to illustrate how  the calculi {\bf G3XD} can be  used to analyse formal representations of the  deontic paradoxes. If Forrester's argument  is formalised as in \cite{M06} then  it does not force to adopt a deontic logic weaker than {\bf KD}. If, instead, it is formalised as in \cite{T97} then it forces the adoption of a logic where $RM$ fails, but the formal derivation  differs substantially from Forrester's informal argument \cite{F84}.
 
\section{Craig's Interpolation Theorem}\label{secinterpol}
In this section we use Maehara's \cite{M60,M61} technique to prove Craig's interpolation theorem for each modal or deontic logic {\bf X} which has  $C$  as theorem only if it has also $M$  (Example \ref{prob} illustrates the problem with the non-standard rule $LR$-$C$).

\begin{theorem}[Craig's interpolation theorem] \label{Craig}
Let $A\to B$ be a theorem of a  logic {\bf X}  that differs from {\bf EC(N)} and its deontic extensions {\bf EC(N)D} and {\bf ECD$^\P$}, then 
there is a formula $I$, which contains  propositional variables common to $A$ and $B$ only, such that both $A\to I$ and $I\to B$ are theorems of {\bf X}.

\end{theorem}

\noindent In order to prove this theorem, we use the following notions.

\begin{definition} 

A \emph{partition} of a sequent $\g\To \d$  is any pair of sequents \\$\<\g_1\To\d_1\;||\;\g_2\To\d_2\>$ such that $\g_1,\g_2=\g$ and $\d_1,\d_2=\d$.\\
A \emph{{\bf G3X}-interpolant of a partition} $\<\g_1\To\d_1\;||\;\g_2\To\d_2\>$ is any formula $I$  such that:
\begin{enumerate}  
\item All propositional variables in $I$ are    in $(\g_1\cup\d_1)\cap(\g_2\cup\d_2)$;
 \item {\bf G3X} $\vdash\g_1\To\d_1,I$ and {\bf G3X} $\vdash I,\g_2\To\d_2$.
 \end{enumerate}

\end{definition}

If $I$ is a {\bf G3X}-interpolant of the partition $\<\g_1\To\d_1\;||\;\g_2\To\d_2\>$, we write

$$(\textrm{{\bf G3X}}\vdash)\;\<\g_1\To\d_1\;\stackrel{I}{||}\;\g_2\To\d_2\>
$$

\noindent where one or more of the multisets $\g_1,\g_2,\d_1,\d_2$ may be empty. When the set of propositional variables in $(\g_1\cup\d_1)\cap(\g_2\cup\d_2)$ is empty, the {\bf X}-interpolant has to be constructed from $\bot$ (and $\top$).
The proof of  Theorem \ref{Craig} is by the following lemma, originally due to Maehara \cite{M60,M61} for (an extension of) classical logic.

\begin{lemma}[Maehara's lemma]\label{Maehara} If {\bf G3X} $\vdash \g\To\d$ and $LR$-$C$ (and $L$-$D^{\P_C}$) is not a  rule of {\bf G3X} (see Tables \ref{modcalculi} and \ref{deoncalculi}), every partition   of $\g\To\d$ has a {\bf G3X}-interpolant.
\end{lemma}

\begin{proof}
The proof is  by induction on the height of the derivation $\de$ of $\g\To\d$. We have to show that each partition of an initial sequent (or of a conclusion of a 0-premiss rule) has a  {\bf G3X}-interpolant and that for each rule of {\bf G3X} (but $LR$-$C$ and $L$-$D^{\P_C}$) we have an effective procedure that outputs a {\bf G3X}-interpolant for any partition of its conclusion from the interpolant(s) of suitable partition(s) of its premiss(es). The proof is modular and, hence, we can consider the modal rules without having to reconsider them in the different calculi.

For the base case of initial sequents with $p$ principal formula, we have four possible partitions, whose interpolants are:\begin{center}\begin{tabular}{ccc}
$(1)\;\<p,\g_1'\To\d_1',p\;\stackrel{\bot}{||}\;\g_2\To\d_2\>\qquad$&$\qquad(2)\;\<p,\g_1'\To\d_1\;\stackrel{p}{||}\;\g_2\To\d'_2,p\>$\\\noalign{\smallskip\smallskip}
$(3)\;\<\g_1\To\d_1',p\;\stackrel{\neg p}{||}\;p,\g'_2\To\d_2\>\qquad$&$\qquad(4)\;\<\g_1\To\d_1\;\stackrel{\top}{||}\;p,\g'_2\To\d'_2,p\>$
\end{tabular}\end{center} 

\noindent and for the base case of rule $L\bot$, we have:
\begin{center}\begin{tabular}{ccc}
$(1)\;\<\bot,\g_1'\To\d_1\;\stackrel{\bot}{||}\;\g_2\To\d_2\>\qquad$&$\qquad(2)\;\<\g_1\To\d_1\;\stackrel{\top}{||}\;\bot,\g'_2\To\d_2,\>$
\end{tabular}\end{center}

 For the proof of (some of) the propositional  cases the reader is referred to \cite[pp. 117-118]{T00}. Thus, we have only to prove that all the modal and deontic  rules of Table \ref{Modal rules} (modulo  $LR$-$C$ and $L$-$D^{\P_C}$)  behave as desired.

\noindent\textbullet$\quad$ {\bf LR-E}$)\quad$ If the last rule applied in $\de$ is 

\begin{center} \begin{prooftree}
 A\To B\qquad B\To A
 \justifies
 \O A,\g\To\d,\O B
 \using LR\textrm{-}E
  \end{prooftree} 
  \end{center}
  \noindent we have  four kinds of partitions of the conclusion:

\begin{center}\begin{tabular}{ccc}
$(1)\;\<\O A,\g_1'\To\d_1',\O B\;||\;\g_2\To\d_2\>\qquad$&$\qquad(2)\;\<\O A,\g_1'\To\d_1\;||\;\g_2\To\d'_2,\O B\>$\\\noalign{\smallskip\smallskip}
$(3)\;\<\g_1\To\d_1',\O B\;||\;\O A,\g'_2\To\d_2\>\qquad$&$\qquad(4)\;\<\g_1\To\d_1\;||\;\O A,\g'_2\To\d'_2,\O B\>$
\end{tabular}\end{center}

\noindent In each  case we have to choose  partitions of the premisses that permit to construct a {\bf G3E(ND)}-interpolant for the partition under consideration.

 In case {\bf(1)} we have 
 $$
\framebox{ \infer[\infrule LR\mbox{-}E]{\<\O A,\g_1'\To\d_1',\O B\;\stackrel{C}{||}\;\g_2\To\d_2\>}{ \<A\To B\;\stackrel{C}{||}\;\To\>&\<B\To A\;\stackrel{D}{||}\;\To\>}
} $$
  
\noindent   This can be shown as follows. By  IH there is some $C$ ($D$) that is a {\bf G3E(ND)}-interpolant of the given partition of the left (right) premiss.  Thus both $C$ and $D$ contains only propositional variables common to $A$ and $B$; and (i) $\vdash A\To B,C\;$ (ii) $\vdash C\To\;$ (iii) $\vdash B\To A,D\;$ and (iv) $\vdash D\To\;$. Since the common language of the partitions of the premisses is empty,   no propositional variable can occur in $C$ nor in $D$. Here is a proof that $C$ is a {\bf G3E(ND)}-interpolant of the partition under consideration (the sequents $A\To B$ and $B\To A$ are derivable since they are the premisses of the instance of $LR$-$E$ we are considering):
  
$$
\infer[\infrule LR\mbox{-}E]{\O A,\g'_1\To\d_1',\O B,C}{\infer{A\To B}{}&\infer{B\To A}{}}
\qquad\qquad
\infer[\infrule LWs+RWs]{C,\g_2\To\d_2}{\infer[\infruler{(ii)}]{C\To}{}}
$$


\noindent In case {\bf(2)} we have

$$
\framebox{\infer[\infrule LR\mbox{-}E]{\<\O A,\g_1'\To\d_1\;\stackrel{\O C}{||}\;\g_2\To\d'_2,\O B\>}{ \<A\To \;\stackrel{C}{||}\;\To B\> &  \<B\To \;\stackrel{D}{||}\;\To A\>}
}$$
  
\noindent By IH it holds that  some $C$ and $D$ are {\bf G3E(ND)}-interpolants of the given partitions of the premisses.  Thus, (i) $\vdash A\To C\;$ (ii) $\vdash C\To B\;$ (iii) $\vdash B\To D\;$ (iv) $\vdash D\To A\;$ and (v)  all propositional variables in $C\cup D$ are in $A\cap B$.
 Here is a proof that $\O C$ is a {\bf G3E(ND)}-interpolant of the given partition (the language condition is satisfied thanks to (v)\,):
 
$$
\infer[\infrule LR\mbox{-}E]{\O A,\g_1'\To\d_1,\O C}{\infer[\infruler{(i)}]{A\To C}{}&
\infer[\infrule Cut]{C\To A}{\infer[\infrule Cut]{C\To D}{\infer[\infruler{(ii)}]{C\To B}{}&\infer[\infruler{(iii)}]{B\To D}{}}&
\infer[\infruler{(iv)}]{D\To A}{}}}
$$

$$
\infer[\infrule LR\mbox{-}E]{\O C,\g_2\To\d_2',\O B}{\infer[\infruler{(ii)}]{C\To B}{}&\infer[\infrule Cut]{B\To C}{\infer[\infruler{(iii)}]{B\To D}{}&\infer[\infrule Cut]{D\To C}{\infer[\infruler{(iv)}]{D\To A}{}&\infer[\infruler{(i)}]{A\To C}{}}}}
$$
%

\noindent  In case {\bf(3)} we have 

$$
\framebox{\infer[\infrule LR\mbox{-}E]{\<\g_1\To\d_1',\O B\;\stackrel{\P C}{||}\;\O A,\g'_2\To\d_2\>}{ \<\To B\;\stackrel{C}{||}\;A\To \> &  \<\To A\;\stackrel{D}{||}\;B\To \>}
}$$
\noindent  By IH, there are $C$ and $D$ that are {\bf G3E(ND)}-interpolants of the partitions of the premisses. Thus (i) $\vdash \To B,C\;$ (ii) $\vdash C,A\To\;$ (iii) $\vdash \To A,D\;$ and (iv) $\vdash D,B\To\;$. We prove that $\P C$  is a {\bf G3E(ND)}-interpolant of the (given partition of the) conclusion as follows:

$$
\infer[\infrule R\neg]{\g_1\To\d_1',\O B,\neg\O\neg C}{\infer[\infrule LR\mbox{-}E]{\O\neg C,\g_1\To\d_1',\O B}{
\infer[\infrule L\neg]{\neg C\To B}{\infer[\infruler{(i)}]{\To B,C}{}}&
\infer[\infrule R\neg]{B\To\neg C}{\infer[\infrule Cut]{B,C\To}{
\infer[\infrule Cut]{C\To D}{\infer[\infruler{(iii)}]{\To D,A}{}&\infer[\infruler{(ii)}]{A,C\To}{}}&
\infer[\infruler{(iv)}]{D,B\To}{}}}}}
$$

$$
\infer[\infrule L\neg]{\neg\O\neg C,\O A,\g_2'\To\d_2}{\infer[\infrule LR\mbox{-}E]{\O A,\g_2'\To\d_2,\O\neg C}{\infer[\infrule R\neg]{A\To \neg C}{\infer[\infruler{(ii)}]{C,A\To}{}}&
\infer[\infrule L\neg]{\neg C\To A}{\infer[\infrule Cut]{\To A,C}{
\infer[\infruler{(iii)}]{\To A,D}{}&
\infer[\infrule Cut]{D\To C}{\infer[\infruler{(ii)}]{\To C,B}{}&\infer[\infruler{(iv)}]{B,D\To}{}}}}}}
$$
%
%
\noindent  In case {\bf(4)} we have 

$$
\framebox{
\infer[\infrule LR-E]{\<\g_1\To\d_1\;\stackrel{C}{||}\;\O A,\g'_2\To\d'_2,\O B\>}{\<\To \;\stackrel{C}{||}\;A\To B\>&\<\To \;\stackrel{D}{||}\;B\To A \>}
}
$$
%
 \noindent By IH, there are  {\bf G3E(ND)}-interpolants $C$ and $D$  of the partitions of the premisses. Thus (i) $\vdash \To C\;$ (ii) $\vdash C,A\To B\;$ (iii) $\vdash \To D\;$ and (iv) $\vdash D,B\To A\;$.  Since the common language of the partitions of the premisses is empty,  no propositional variable  occurs in $C$ nor in $D$. We  show that $C$ is a {\bf G3E(ND)}-interpolant of the partition under consideration as follows (as in case {\bf (1)}, $A\To B$ and $B\To A$, being the premisses of the instance of $LR$-$E$ under consideration, are derivable):
 
$$
\infer[\infrule LWs+RWs]{\g_1\To\d_1,C}{\infer[\infruler{(i)} ]{\To C}{}}
\qquad
\infer[\infrule LR\mbox{-}E]{C,\O A,\g_2'\To\d_2',\O B}{\infer{A\To B}{}&\infer{B\To A}{}}
$$

\noindent\textbullet$\quad$ {\bf LR-M}$)\quad$ If the last rule applied in $\de$ is 

$$
\infer[\infrule LR\mbox{-}M]{\O A,\g\To\d,\O B}{A\To B}
$$
%

\noindent we give directly the {\bf G3M(ND)}-interpolants of the possible  partitions  of the conclusion (and of the appropriate partition of the  premiss). The proofs are parallel to those for $LR$-$E$. 

$$
\framebox{ \infer[\infrule LR\mbox{-}M]{\<\O A,\g_1'\To\d_1',\O B\;\stackrel{C}{||}\;\g_2\To\d_2\>}{\<A\To B\;\stackrel{C}{||}\;\To\> } 
\qquad
 \infer[\infrule LR\mbox{-}M]{\<\O A,\g_1'\To\d_1\;\stackrel{\O C}{||}\;\g_2\To\d'_2,\O B\>}{\<A\To \;\stackrel{C}{||}\;\To B\>}
 }
$$
$$\framebox{
\infer[\infrule LR\mbox{-}M]{\<\g_1\To\d_1',\O B\;\stackrel{\P C}{||}\;\O A,\g'_2\To\d_2\>}{\<\To B\;\stackrel{C}{||}\;A\To \>} 
\qquad
\infer[\infrule LR\mbox{-}M]{\<\g_1\To\d_1\;\stackrel{C}{||}\;\O A,\g'_2\To\d'_2,\O B\>}{ \<\To \;\stackrel{C}{||}\;A\To B\> }
}$$
\noindent\textbullet$\quad$ {\bf LR-R}$)\quad$ If the last rule applied in $\de$ is 

$$
\infer[\infrule LR\mbox{-}R]{\O A,\O \Pi,\g\To\d,\O B}{A,\Pi\To B}
$$
  \noindent we have  four kinds of partitions of the conclusion:

\begin{center}\begin{tabular}{ll}
$(1)\qquad$&$\<\O A,\O\Pi_1,\g_1'\To\d_1',\O B\;||\;\O\Pi_2,\g'_2\To\d_2\>\quad$\\\noalign{\smallskip\smallskip}
$(2)$&$\<\O A,\O\Pi_1,\g_1'\To\d_1\;||\;\O\Pi_2,\g'_2\To\d'_2,\O B\>$\\\noalign{\smallskip\smallskip}
$(3)$&$\<\O\Pi_1,\g'_1\To\d_1',\O B\;||\;\O A,\O\Pi_2,\g'_2\To\d_2\>\quad$\\\noalign{\smallskip\smallskip}
$(4)$&$\<\O\Pi_1,\g'_1\To\d_1\;||\;\O A,\O\Pi_2,\g'_2\To\d'_2,\O B\>$
\end{tabular}\end{center}


 In case {\bf(1)} we have two subcases according to whether $\Pi_2$ is empty or not. If it is not empty we have 
%
%

%
$$
\framebox{\infer[\infrule LR\mbox{-}R]{\<\O A,\O\Pi_1,\g_1'\To \d_1',\O B\;\stackrel{\P C}{||}\; \O\Pi_2,\g_2'\To\d_2\> }{\<A,\Pi_1\To B\;\stackrel{C}{||}\; \Pi_2\To\;\> }
}$$
  
\noindent  By IH, there is  a {\bf G3R(D$^\star$)}-interpolant $C$ of the chosen partition of the premiss. Thus (i) $\vdash A,\Pi_1\To B,C$ and (ii) $\vdash C,\Pi_2\To$, and we have the following derivations

$$
\infer[\infrule R\neg]{\O A,\O\Pi_1,\g_1'\To\d_1',\O B,\neg\O\neg C}{\infer[\infrule LR\mbox{-}R]{\O\neg C,\O A,\O\Pi_1,\g_1'\To\d_1',\O B}{\infer[\infrule L\neg]{\neg C,A,\Pi_1\To B}{\infer[\infruler{(i)}]{A,\Pi_1\To B,C}{}}}}
\qquad
\infer[\infrule L\neg]{\neg\O\neg C,\O\Pi_2,\g'_2\To\d_2}{\infer[\infrule LR\mbox{-}R]{\O\Pi_2,\g'_2\To\d_2,\O \neg C}{\infer[\infrule R\neg]{\Pi_2\To \neg C}{\infer[\infruler{(ii)}]{C,\Pi_2\To}{}}}}
$$

\noindent When $\Pi_2$ (and $\O \Pi_2$)  is empty we cannot proceed as above since we cannot apply $LR$-$R$  in the right derivation. But in this case, reasoning like in case {\bf (1)} for rule $LR$-$E$, we can show that

$$
\framebox{\infer[\infrule LR\mbox{-}R]{\<\O A,\O\Pi_1,\g_1'\To \d_1',\O B\;\stackrel{C}{||}\; \g_2'\To\d_2\> }{\<A,\Pi_1\To B\;\stackrel{C}{||}\; \To\;\> }
}$$


\noindent Cases {\bf(2)} and {\bf(3)} are similar to  the corresponding cases for rule $LR$-$E$:
{\footnotesize$$
\framebox{\infer[\infrule LR\mbox{-}R]{\<\O A,\O\Pi_1,\g'_1\To\d_1\;\stackrel{\O C}{||}\;\O\Pi_2,\g'_2\To\d'_2,\O B\>}{\< A,\Pi_1\To\;\stackrel{C}{||}\;\Pi_2\To B\>}\quad
%
%
\infer[\infrule LR\mbox{-}R]{\<\O\Pi_1,\g'_1\To\d'_1\,\O B\;\stackrel{\P C}{||}\;\O A,\O\Pi_2,\g'_2\To\d_2\>}{\<\Pi_1\To B\;\stackrel{C}{||}\; A,\Pi_2\To\;\>}
}$$}

 In case {\bf(4)} we have two subcases according to whether $\Pi_1$ is empty or not:
{\small$$
\framebox{\infer[\infrule LR\mbox{-}R]{\<\g'_1\To\d_1\;\stackrel{C}{||} \O A,\O\Pi_2,\g_2'\To\d_2',\O B\>}{\<\;\To\;\stackrel{C}{||}A,\Pi_2\To B\>}\quad
\infer[\infrule LR\mbox{-}R]{\<\O\Pi_1,\g'_1\To\d_1\;\stackrel{\O C}{||} \O A,\O\Pi_2,\g_2'\To\d_2',\O B\>}{\<\Pi_1\To\;\stackrel{C}{||}A,\Pi_2\To B\>}
}$$ }
The proofs are similar to those for case {\bf (1)}.
%
%
%
%

\noindent\textbullet$\quad$ {\bf LR-K}$)\quad$ If the last rule applied in $\de$ is 

$$
\infer[\infrule LR\mbox{-}K]{\O \Pi,\g\To\d,\O B}{\Pi\To B}
$$
\noindent we give directly the {\bf G3K(D)}-interpolants of the two  possible  partitions  of the conclusion:

%
%
{\small$$
\framebox{\infer[\infrule LR\mbox{-}K]{\<\O\Pi_1,\g_1'\To\d_1\;\stackrel{\O C}{||}\;\O\Pi_2,\g_2'\To\d_2',\O B\>}{\<\Pi_1\To\;\stackrel{C}{||}\;\Pi_2\To B\>}\qquad
\infer[\infrule LR\mbox{-}K]{\<\O\Pi_1,\g_1'\To\d_1',\O B\;\stackrel{\P C}{||}\;\O\Pi_2,\g_2'\To\d_2\>}{\< \Pi_1\To B\;\stackrel{C}{||}\;\Pi_2\To\;\>}
}$$}
The proofs are, respectively, parallel to those for cases {\bf (2)} and {\bf(3)} of $LR$-$E$ (when $\Pi=\emptyset$, we can proceed as for rule $R$-$N$ and use $C$ instead of $\O C$ and of $\P C$, respectively).


\noindent\textbullet$\quad$ {\bf L-D}$^\bot)\quad$ If the last rule applied in $\de$ is 

\begin{center} \begin{prooftree}
 A\To 
 \justifies
 \O A,\g\To\d
 \using \infrule{L\mbox{-}D^\bot}
  \end{prooftree} 
  \end{center}
  \noindent we have  two kinds of partitions of the conclusion, whose {\bf G3XD$^\bot$}-interpolants are, respectively:

%
 \begin{center}\framebox{ \begin{prooftree}
\<A\To\;\stackrel{C}{||}\;\To\;\>
 \justifies
\<\O A,\g_1'\To\d_1\;\stackrel{C}{||}\;\g_2\To\d_2\>
 \using \infrule{L\mbox{-}D^\bot}
  \end{prooftree}\qquad
   \begin{prooftree}
\<\;\To\;\stackrel{C}{||}\; A\To\;\>
 \justifies
\<\g_1\To\d_1\;\stackrel{C}{||}\; \O A,\g_2'\To\d_2\>
 \using \infrule{L\mbox{-}D^\bot}
  \end{prooftree} }
  \end{center}
  
%
%
%

\noindent\textbullet$\quad$ {\bf L-D$^\P$}$)\quad$ If the last rule applied in $\de$ is 

$$
\infer[\infrule L\mbox{-}D^{\P_E}]{\O A,\O B,\g\To \d}{A,B\To\qquad&\To A,B }\qquad
\textnormal{or}\qquad
\infer[\infrule L\mbox{-}D^{\P_M}]{\O A,\O B,\g\To \d}{A,B\To}
$$
  \noindent we have  three kinds of partitions of the conclusion:

\begin{center}\begin{tabular}{ll}
$(1)\qquad$&$\<\O A,\O B,\g'_1\To\d_1\;||\;\g_2\To\d_2\>\quad$\\\noalign{\smallskip\smallskip}
$(2)$&$\<\g_1\To\d_1\;||\;\O A,\O B,\g'_2\To\d_2\>$\\\noalign{\smallskip\smallskip}
$(3)$&$\<\O A,\g'_1\To\d_1\;||\;\O B,\g'_2\To\d_2\>\quad$\\\noalign{\smallskip\smallskip}
\end{tabular}\end{center}

 In cases {\bf (1)} and {\bf (2)} we have, respectively (omitting the right premiss for $L$-$D^{\P_M}$):
 
 $$
\framebox{\infer[\infrule L\mbox{-}D^\P]{\<\O A,\O B,\g_1'\To \d_1\;\stackrel{ C}{||}\; \g_2\To\d_2\> }{\<A,B\To \;\stackrel{C}{||}\; \To\;\> \qquad \<\To \;\stackrel{D}{||}\; \To \;A,B\>}
\infer[\infrule L\mbox{-}D^\P]{\<\g_1\To \d_1\;\stackrel{ C}{||}\; \O A,\O B,\g'_2\To\d_2\> }{\<A,B\To \;\stackrel{C}{||}\; \To\;\> \qquad \<\To \;\stackrel{D}{||}\; \To \;A,B\>}
}$$

Finally, in case {\bf (3)} we have:

 $$
\framebox{\infer[\infrule L\mbox{-}D^\P]{\<\O A,\g_1'\To \d_1\;\stackrel{ \O C}{||}\; \O B,\g'_2\To\d_2\> }{\<A\To \;\stackrel{C}{||}\; B \To\;\> \qquad \<\To \; A\stackrel{D}{||}\; \To \;B\>}
}$$
By IH, we can assume that $C$ is an interpolant of the partition of the left premiss and $D$ of the right one.
We have the following {\bf G3YD$^\P$}-derivations ({\bf Y} $\in\{$ {\bf E,M}$\}$):

$$
\infer[\infrule LR\mbox{-}E]{\O A,\g_1'\To\d_1,\O C}{\infer[\infrule IH]{A\To C}{}&
\infer[\infrule Cut]{C\To A}{\infer[\infrule IH]{\To A,D}{}&
\infer[\infrule Cut]{D,C\To}{
\infer[\infrule IH]{D\To B}{}&\infer[\infrule IH]{B,C\To}{}
}}}
$$

$$
\infer[\infrule L\mbox{-}D^{\P_{E}}]{\O C,\O B,\g_2'\To\d_2}{
\infer[\infrule IH]{C,B\To}{}&\infer[\infrule Cut]{\To C,B}{
\infer[\infrule Cut]{\To C,D}{
\infer[\infrule IH]{\To D,A}{}&
\infer[\infrule IH]{A\To C}{}}&
\infer[\infrule IH]{D\To B}{}}
}
$$
It is also immediate to notice that  $\O C$ satisfies the language condition for being a {\bf G3YD$^\P$}-interpolant of the conclusion since, by IH, we know that each propositional variable occurring in $C$ occurs   in $A\cap B$.

\noindent\textbullet$\quad$ {\bf L-D$^\star$}$)\quad$ If the last rule applied in $\de$ is 

\begin{center} \begin{prooftree}
 \Pi\To 
 \justifies
 \O\Pi, \g\To\d
 \using \infrule{L\mbox{-}D^\star}
  \end{prooftree} 
  \end{center}
  \noindent we have  the following kind of  partition:
\quad$\< \O\Pi_1,\g'_1\To\d_1\;{||}\;\O\Pi_2,\g'_2\To\d_2\>$

If $\Pi_1$ is not empty we have: 
\begin{center} \framebox{\begin{prooftree}
\<\Pi_1\To\;\stackrel{C}{||}\;\Pi_2\To\;\>
 \justifies
\< \O\Pi_1,\g'_1\To\d_1\;\stackrel{\O C}{||}\;\O\Pi_2,\g'_2\To\d_2\>
 \using  \infrule{L\mbox{-}D^\star}
  \end{prooftree} }
  \end{center}

\noindent By IH, there is some  $C$ that is  an interpolant of the premiss. It holds  that $\vdash \Pi_1\To C$ and $\vdash C,\Pi_2\To\;$. We  show that $\O C$ is a {\bf G3YD}-interpolant ({\bf Y} $\in\{${\bf R,K}$\}$) of the partition of the conclusion as follows: 

$$
\infer[\infrule LR\mbox{-}Y]{\O\Pi_1,\g_1'\To\d_1,\O C}{\infer[\infruler{IH}]{\Pi_1\To C}{}}
\qquad
\infer[\infrule L\mbox{-}D^*]{\O C,\O \Pi_2,\g_2'\To\d_2}{\infer[\infruler{IH}]{C,\Pi_2\To}{}}
$$
If, instead, $\Pi_1$ is empty then $\Pi_2$ cannot be empty and we have:
\begin{center} \framebox{\begin{prooftree}
\<\;\To\;\stackrel{C}{||}\;\Pi_2\To\;\>
 \justifies
\< \g_1\To\d_1\;\stackrel{\P C}{||}\;\O\Pi_2,\g'_2\To\d_2\>
 \using \infrule{ L\mbox{-}D^\star}
  \end{prooftree} }
  \end{center} 
  
\noindent   By IH there is a formula $C$, containing no propositional variable, such that  $\vdash \;\To C$ and $\vdash C,\Pi_2\To\;$ . Thus, {\bf G3YD} $\vdash\g_1\To\d_1,\P C$ ($L$-$D^*$ makes $\To\P C$ derivable from $\To C$) and  {\bf G3YD} $\vdash\P\top,\O\Pi_2,\g_2'\To\d_2$ ($LR$-$Y$ makes  $\P C,\O\Pi_2\To$ derivable from $C,\Pi_2\To$ when $\Pi_2\neq \emptyset$).
%

 \noindent\textbullet$\quad$ {\bf R-N}$)\quad$ If the last rule applied in $\de$ is 

\begin{center} \begin{prooftree}
 \To A 
 \justifies
\g\To\d,\O A
 \using \infrule{R\mbox{-}N}
  \end{prooftree} 
  \end{center}
%

\noindent The interpolants for the two possible partitions are:

\noindent\framebox{ \begin{tabular}{cccc}
 $(1)\;$& \begin{prooftree}
\<\;\To A\stackrel{\bot}{||}\;\To\;\>
 \justifies
\< \g_1\To\d_1',\O A\;\stackrel{\bot}{||}\;\g_2\To\d_2\>
 \using \infrule{R\mbox{-}N}\quad
  \end{prooftree} &

%

$(2)\;$& \begin{prooftree}
\<\;\To \;\stackrel{\top}{||}\;\To A\>
 \justifies
\< \g_1\To\d_1\;\stackrel{\top}{||}\;\g_2\To\d_2',\O A\>
 \using \infrule{R\mbox{-}N}
  \end{prooftree} 
  \end{tabular}}\vspace{0.3cm}

%

\noindent This completes the proof.{}
\end{proof}

\begin{proof}[Proof of Theorem \ref{Craig}]
Assume that $A\to B$ is a theorem of {\bf X}. By Theorem \ref{comp} and Lemma \ref{inv} we have that {\bf G3X} $\vdash A\To B$. By Lemma \ref{Maehara}  (taking $A$ as $\g_1$ and  $B$ as $\d_2$ and $\g_2,\d_1$ empty) and Theorem \ref{comp} there exists  a formula $I$ that is an interpolant of $A\to B$ -- i.e. $I$ is such such that  all propositional variables occurring in $I$, if any, occur in both  $A$ and $B$ and such that  $A\to I$ and $I\to B$ are theorems of {\bf X}.{}
\end{proof}

Observe that the proof is constructive in that  Lemma \ref{Maehara} gives a procedure to  extract  an interpolant for $A\to B$ from a given derivation of $A\To B$. Furthermore the proof is purely proof-theoretic in that it makes no use of model-theoretic notions.

Craig's theorem is often -- e.g., in \cite{M61} for an extension of  classical   logic -- stated in the following stronger  version:
\begin{quote}
If $A\to B$ is a theorem of the logic {\bf X}, then 
\begin{enumerate}
\item If $A$ and $B$ share some propositional variable, there is a formula $I$, which contains  propositional variables common to $A$ and $B$ only, such that both $A\to I$ and $I\to B$ are theorems of {\bf X};
\item Else, either $\neg A$  or $B$ is a theorem of {\bf X}.
\end{enumerate}
\end{quote}

\noindent But the second condition  doesn't hold for   modal and deontic logics where at least one of $N:=\O\top$ and $D^\bot:=\P\top$ is not a theorem. 
To illustrate,  it holds that  $ \O\top\to\O\top$ is a theorem of {\bf E} and its interpolant is $\O\bot$ (see Figure \ref{fig}), but neither  $\neg \O\top$ nor $\O\top$ is a theorem of {\bf E}. Analogously, we have that $\O\bot\to\O\bot$ is a theorem of  {\bf E} and its interpolant is $\O\bot$ (see Figure \ref{fig}), but neither $\neg \O\bot$ nor $\O\bot$ is a  theorem of {\bf E}. These counterexamples work in all extensions of {\bf E} that don't have both $N$ and $D^\bot$ as theorems: to prove the stronger version of Craig's theorem we need $N$ and $D^\bot$, respectively.

\begin{figure}
\scalebox{0.9900}{\infer[\infrule LR\mbox{-}E]{\< \;\O\top\To\;\stackrel{\O\top}{||}\;\To\O\top\;\>}{
\< \;\top\To\;\stackrel{\top}{||}\;\To\top\;\>& \< \;\top\To\;\stackrel{\top}{||}\;\To\top\;\>}\qquad
\infer[\infrule LR\mbox{-}E]{\< \;\O\bot\To\;\stackrel{\O\bot}{||}\;\To\O\bot\;\>}{
\< \;\bot\To\;\stackrel{\bot}{||}\;\To\bot\;\>& \< \;\bot\To\;\stackrel{\bot}{||}\;\To\bot\;\>}
}

\caption{Construction of an {\bf ED}-interpolant for   $\O\top\to\O\top$ and for $\O\bot\to\O\bot$}\label{fig}
\end{figure}


Among the deontic logics considered here, the stronger version of Craig's theorem holds only for {\bf END$^{\bot(\P)}$},  {\bf MND$^{\bot(\P)}$}, and {\bf KD}, as shown by the following 

\begin{corollary}\label{cor} Let {\bf XD} be one of {\bf END$^{\bot(\P)}$},  {\bf MND$^{\bot(\P)}$}, and {\bf KD}. If $A\to B$ is a theorem of {\bf XD} and $A$ and $B$ share no propositional variable, then either $\neg A$ or $B$ is a theorem of {\bf XD}.
\end{corollary}

\begin{proof}
Suppose that {\bf XD} $\vdash A\to B$ and that $A$ and $B$ share no propositional variable, then the interpolant $I$  is constructed from $\bot $ and $\top$ by means of classical and deontic operators.  Whenever  $D^\bot$ and $N$ are theorems of {\bf XD}, we have that $\P\top\leftrightarrow \top$, $\O\top\leftrightarrow \top$,  $\P\bot\leftrightarrow \bot$, and $\O\bot\leftrightarrow \bot$ are theorems of {\bf XD}.  Hence, the interpolant $I$ is (equivalent to) either $\bot$ or $\top$. In the first case {\bf XD} $\vdash\neg A$ and in the second one {\bf XD} $\vdash B$.{}
\end{proof}

\noindent As noted in \cite[p. 298]{F83}, Corollary \ref{cor} is a Halld\'en-completeness  result. A logic {\bf X} is \emph{Halld\'en-complete} if, for every formulas $A$ and $B$ that share no propositional variable, {\bf X} $\vdash A\lor B$ if and only if {\bf X} $\vdash A$ or {\bf X} $\vdash B$. All the modal and deontic logics considered here, being based on classical logic, are such that $A\to B$ is equivalent to $\neg A \lor B$. Thus  the deontic logics considered in Corollary \ref{cor} are Halld\'en-complete, whereas all other non-normal logics for which we have proved interpolation  are Halld\'en-incomplete since they don't satisfy Corollary \ref{cor}.

\begin{example}[Maehara's lemma and rule $LR$-$C$]\label{prob}
 We have not been able to prove Maehara's Lemma \ref{Maehara} for rule $LR$-$C$ because of the cases where the principal formulas of the antecedent are splitted in the two elements of the partition. In particular, if we have two principal formulas in the antecedent, the problematic partitions  are (omitting the weakening contexts):
 \begin{center}
 (1)\quad $\<\O A_1\To\;||\;\O A_2\To \O B\>$\qquad\qquad
 (2)\quad $\<\O A_1\To\O B\;||\;\O A_2\To\>$
 \end{center}
To illustrate, an interpolant of the first partition would be a formula $I$ such that:

$$
 (i)\quad \vdash \O A_1\To I\qquad (ii)\quad\vdash I,\O A_2\To\O B\qquad(iii)\quad p\in I\textnormal{ only if }p\in (A_1)\cap(A_2,B)
$$
But we have not been able to find partitions of the premisses allowing to find such  $I$. More in details, for the first premiss it is natural to   consider the partition $\< A_1\To\;\stackrel{C}{||}\; A_2\To  B\>$ in order to find an $I$ that satisfies $(iii)$.  But, for any combination of the partitions of the other two premisses that is compatible with $(iii)$, we can prove  that   $(ii)$ is satisfied (by $\O C$) but we have not been able to prove that also $(i)$ is satisfied.
\end{example}

\section{Conclusion}\label{conc}
We presented cut- and contraction-free sequent calculi for non-normal modal and deontic logics. We have proved that these calculi have good structural properties in that weakening and contraction are height-preserving admissible and cut is (syntactically) admissible. Moreover, we have shown that these calculi allow for a terminating decision procedure whose complexity is in {\sc Pspace}. Finally, we have given a constructive proof of Craig's interpolation property for all the logics that do not contain  rule $LR$-$C$. As far as we know, it is still an open problem whether it is possible to give a constructive proof of interpolation for these logics. Another open question is whether the calculi given here can be used to give a constructive proof of  the uniform interpolation property for non-normal logics as it is done in \cite{P92} for $\mathbf{IL_p}$ and in \cite{B07} for  {\bf K} and {\bf T}. 

\vspace{0.3cm}

\noindent {\bf Thanks.} Thanks are due to Tiziano Dalmonte, Simone Martini,  and two anonymous referees for many helpful suggestions.
%


%

\noindent  {\sc Eugenio Orlandelli} \\
 Department of Philosophy and Communication Studies\\ University of Bologna\\
 Via Zamboni 38\\
 I-40126 Bologna, Italy\\
{\tt eugenio.orlandelli@unibo.it}

 \end{document}